\documentclass[11pt]{article}

\usepackage[margin=2.4cm]{geometry}
\usepackage{amsmath,amssymb}

\usepackage{amsthm}

\numberwithin{equation}{section}

\usepackage{enumitem,color}
\usepackage{hyperref}
\usepackage{multirow,booktabs}

\usepackage{graphicx}

\newtheorem{theorem}{Theorem}[section]
\newtheorem{lemma}[theorem]{Lemma}
\newtheorem{remark}[theorem]{Remark}

\newcommand{\p}{\partial}

\title{The optimal error analysis of nonuniform L1 method for the variable-exponent subdiffusion model}

\author{
Wenlin Qiu\thanks{School of Mathematics, Shandong University, Jinan 250100, China. Email: qwllkx12379@163.com.}
\and Kexin Li\thanks{Corresponding author. School of Statistics and Mathematics, Yunnan University of Finance and Economics, Kunming 650221, China. Email: likx1213@163.com.}
\and Yiqun Li\thanks{Corresponding author. School of Mathematics and Statistics, Wuhan University, Wuhan 430072, China. Email: YiqunLi24@outlook.com.}
\and Hao Zhang\thanks{School of Computer Science and Engineering, Sun Yat-sen University, Guangzhou 510006, Guangdong, China. Email: zhangh925@mail2.sysu.edu.cn.}
}

\date{}

\begin{document}

\maketitle

\begin{abstract}
This work investigates the optimal error estimate of the fully discrete scheme for the variable-exponent subdiffusion model under the nonuniform temporal mesh. We apply the perturbation method to reformulate the original model into its equivalent form, and apply the L1 scheme as well as the interpolation quadrature rule to discretize the Caputo derivative
term and the convolution term in the reformulated model, respectively.
We then prove the temporal convergence rates $O\left(N^{-\min\{2-\alpha(0), \, r\alpha(0)\}} \right)$ under the nonuniform mesh,  which improves the existing convergence results in [Zheng, {\it CSIAM T. Appl. Math.} 2025] for $r\geq \frac{2-\alpha(0)}{\alpha(0)}$. Numerical results are presented to substantiate the theoretical findings.

\end{abstract}

\section{Introduction}

We consider the following variable-exponent subdiffusion problem, which describes the transient dispersion in the highly heterogeneous porous medium \cite{JiSun,FanHu,FuZhu,SunCha,VanHen,ZhengMMS}
\begin{align}
   & \partial_t^{\alpha(t)} u(x,t) - \Delta u(x,t)  = f(x,t), \quad (x,t) \in \Omega \times (0, T], \label{eq1.1} \\
   & u(x,0) = u_0(x), \quad x \in \Omega, \quad  u(x,t) = 0, \quad (x,t) \in \partial\Omega \times (0, T], \label{eq1.2}
\end{align}
where $\Omega \subset \mathbb{R}^{d}$ ( $1 \le d \le 3$) is a bounded domain with smooth boundary $\partial\Omega$, and the functions $u_{0}$ and $f$ are prescribed. Here the variable-exponent Caputo fractional derivative  with the variable exponent $ 0< \alpha(t) <1$ is defined by \cite{GarGiu,Jin,LorHar,Pod}
\begin{align*}
    \partial_t^{\alpha(t)} v(x,t)
    := \beta_{1-\alpha(t)} * \partial_t v = \int_{0}^{t} \frac{(t-s)^{-\alpha(t-s)}}{\Gamma(1-\alpha(t-s))} \p_sv (x,s)ds,
\end{align*}
where $*$ denotes the convolution, $\beta_{\nu} := t^{\nu-1}/\Gamma(\nu)$, and $\Gamma(\cdot)$ is the Euler Gamma function.

From the perspective of numerical analysis, the convolution kernel $\beta_{1-\alpha(t)}$ in \eqref{eq1.1} which is neither positive definite nor monotone brings essential challenges. To circumvent this issue, the work  \cite{Zheng} proposed two methods summarized in Figure \ref{ModelRef}, namely, the perturbation method (Transform 1) and the convolution method (Transform 2), to reformulate model \eqref{eq1.1}--\eqref{eq1.2} into its equivalent forms to carry out the analysis.
\begin{figure}[ht]\label{ModelRef}
\center
	\hspace{-2cm}
\includegraphics[page=1,trim=0cm 8cm 4.5cm 3cm,clip,width=1\textwidth]{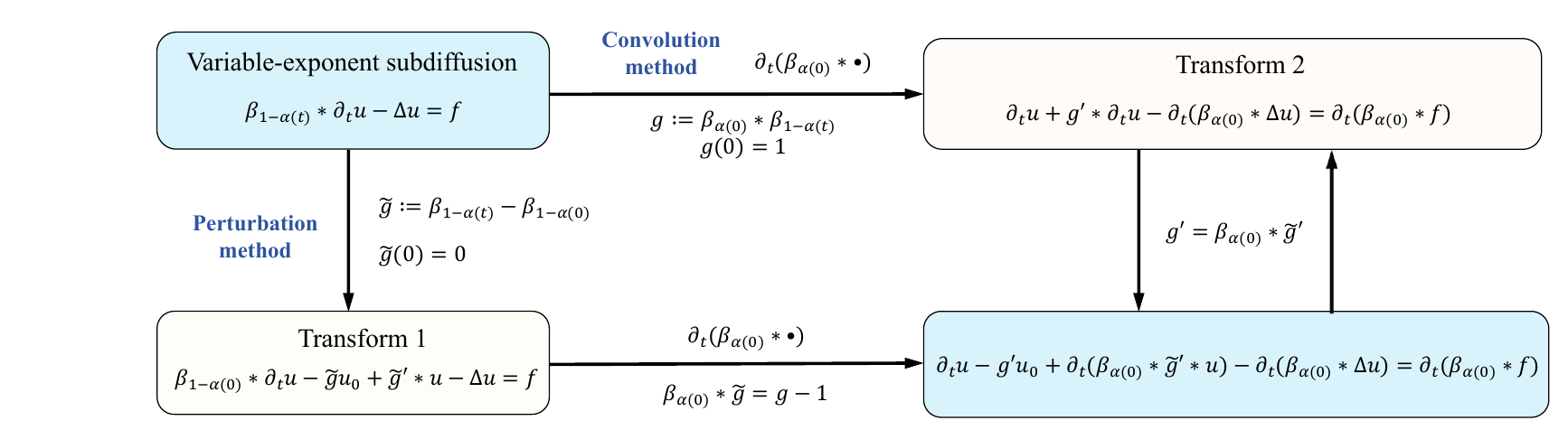}
\caption{The equivalent models of equation \eqref{eq1.1} via perturbation method and convolution method.} 
\end{figure}
In addition, the numerical analysis based on the reformulated model (Transform 2 in Figure \ref{ModelRef}) is conducted, which induces an integro-differential equation after integration with respect to the time variable $t$. The temporal $L^2(0,T)$ convergence result under the uniform temporal mesh was established with the time convergence rate $O(N^{-(\frac{3}{2}\alpha_0+\frac{1}{2})})$ and $\alpha_0 : = \alpha(0)$, which is not optimal since as $\alpha_0 \to 0^+$, the resulting convergence rate converges to  $\frac{1}{2}$. In this case,  the convergence rate is strictly below first order and therefore not satisfactory. This motivates the current investigations on the optimal error estimate of the variable-exponent subdiffusion model \eqref{eq1.1}--\eqref{eq1.2}.

To compensate for this gap, we adopt the Transform 1 in Figure~\ref{ModelRef} to carry out the numerical analysis
\begin{align}\label{VEPDE}
  & \partial_t^{\alpha_0}  u(x,t)
 + (\widetilde{g}' * u) (x,t)
- \Delta u(x,t)
 = F(x,t), \quad (x,t) \in \Omega \times (0, T];   \\
 & u(x,0) = u_0(x), \quad x \in \Omega, \quad  u(x,t) = 0, \quad (x,t) \in \partial\Omega \times (0, T], \nonumber
\end{align}
where $F:=f + \widetilde{g} \, u_0$ and the kernel function $\widetilde{g}$ is defined as follows
\begin{equation}\label{g}
\widetilde{g}(t)
= \int_0^t \frac{t^{-\alpha(z)}}{\Gamma(1-\alpha(z))}
\Big[ -\alpha'(z)\ln t + \bar{\gamma}\, (1-\alpha(z)) \, \alpha'(z) \Big]\, dz.
\end{equation}
 Here $\bar{\gamma}(z):=\Gamma'(z)/\Gamma(z)$ and $\widetilde{g}$ defined in \eqref{g}
as well as its derivative could be bounded by
% . We then need to give the property of $\widetilde{g}$ for subsequent analysis. We employ
% \[
% t^{-\alpha(z)}
% = t^{-\alpha_0} e^{[\alpha_0 - \alpha(z)] \ln t}
% \le t^{-\alpha_0} e^{Cz|\ln t|}
% \le t^{-\alpha_0} e^{Ct |\ln t|}
% \le C t^{-\alpha_0}
% \]
% to estimate $\widetilde{g}(t)$ as
% \[
% |\widetilde{g}(t)|
% \le C \int_0^t z^{-\alpha_0}(1+|\ln t|)\, dz
% \le C t^{1-\alpha_0}(1+|\ln t|)
% \to 0 \quad \text{as } t\to 0,
% \]
% and bound $\widetilde{g}'(t)$ as
\begin{align}\label{qq1}
 |\widetilde{g}(t)|
\to 0 \quad \text{as } t\to 0,\quad |\widetilde{g}'(t)| \le Ct^{-\alpha^*}, \quad 0<\alpha_0<\alpha^*<1.
\end{align}

In this work, we follow \cite{Zheng} to assume  the  following regularity results   for problem \eqref{eq1.1}--\eqref{eq1.2}
\begin{align}\label{regu}
    \|\partial_t \Delta u\| \leq Ct^{\alpha_0-1}, \quad  \|\partial_t^2 \Delta u\| \leq Ct^{\alpha_0-2}, \quad t \in (0, T],
\end{align}
where the positive constant  $C$ denotes a generic positive constant independent of the space-time step sizes, whose value may vary at different occurrences. Our \emph{main contributions} are summarized as follows:
\begin{itemize}
\item We apply the L1 scheme and interpolation quadrature rule to discretize the Caputo derivative term and the convolution term in the reformulated model \eqref{VEPDE}, and follow \cite{Liao,Liao1} to apply the discrete convolution kernel to analyze the time-discrete scheme, based on which the error estimate of the fully discrete Galerkin scheme is established.

\item Based on the solution regularity assumptions \eqref{regu}, we prove the temporal convergence rates $O\left(N^{-\min\{2-\alpha_0, \, r\alpha_0\}} \right)$ under the nonuniform mesh, which exactly coincides with the convergence results for the constant-exponent subdiffusion models in \cite{Liao,Stynes}. For $r\geq \frac{2-\alpha_0}{\alpha_0}$,    the derived convergence rate reaches $O(N^{-(2-\alpha_0)})$, which   improves the existing result in \cite{Liu,Zheng}. 
\end{itemize}

The remainder of this paper is organized as follows.
In  Section~\ref{sec3}, we develop and analyze the stability of the time-semidiscrete scheme for the reconstructed problem. The error analysis of the time-semidiscrete scheme and the fully discrete finite element scheme is analyzed in Section~\ref{sec4}. In Section~\ref{sec6},  we provide some numerical examples to validate the theoretical results. The concluding remarks are addressed in the last section.

\section{Stability of the time-semidiscrete scheme}\label{sec3}

In this section, we shall establish and analyze the stability of the time-semidiscrete scheme for the problem \eqref{VEPDE}.

\subsection{Establishment of time-semidiscrete scheme}

Here, we introduce several notations used throughout the analysis. Let $N$ be a positive integer and define the nonuniform time levels \cite{Kop,KopMen,Quan}
\begin{equation}\label{mesh}
      t_n = T\Big(\frac{n}{N}\Big)^{r}=(n\tau)^r, \quad \tau=\frac{T^{1/r}}{N}, \quad  0\le n\le N 
\end{equation}
 with the grading index $r\ge 1$, and the step sizes $\tau_n=t_n-t_{n-1}$. Let $u^n=u(x,t_n)$. Evaluating \eqref{VEPDE} at $t=t_n$ for $1\le n\le N$ gives
\begin{equation}\label{qwl03_new}
  \partial_t^{\alpha_0}u(x,t_n)-\Delta u(x,t_n)
  = -(\widetilde g' * u)(x,t_n)+F(x,t_n).
\end{equation}
The Caputo derivative is approximated by the nonuniform L1 formula \cite{JinLiZouSINUM,Liao,Stynes}
\begin{equation}\label{qwl04_new}
  D_N^{\alpha_0}u^n
  =\sum_{k=1}^{n}\int_{t_{k-1}}^{t_k}
     \beta_{1-\alpha_0}(t_n-s)\frac{\delta_\tau u^k}{\tau_k}\,ds
  =\sum_{k=1}^{n} a_{n-k}^{(n)}\delta_\tau u^k ,
\end{equation}
where $\delta_\tau u^k:=u^{k}-u^{k-1}$ and the weights
\begin{equation*}
  a_{n-k}^{(n)}
  =\int_{t_{k-1}}^{t_k}\frac{\beta_{1-\alpha_0}(t_n-s)}{\tau_k}\,ds
  =\frac{\beta_{2-\alpha_0}(t_n-t_{k-1})-\beta_{2-\alpha_0}(t_n-t_k)}{\tau_k}, \quad 1\leq k \leq n
\end{equation*}
form a nonincreasing sequence. Under the regularity assumption \eqref{regu}, the L1 truncation error satisfies (see \cite{Stynes})
\begin{equation}\label{qwl06_new}
  \|(R_1)^n\|
  =\|D_N^{\alpha_0}u^n-\partial_t^{\alpha_0}u(\cdot,t_n)\|
  =O\!\left(n^{-\min\{2-\alpha_0,\; r\alpha_0\}}\right),
  \quad n\ge 1.
\end{equation}
 We use the interpolation quadrature rule to approximate the integral term on the right-hand side of \eqref{qwl03_new}
\begin{equation}\label{qwl07}
\begin{split}
     (\tilde{k} * u) (x,t_n) & = \int_{0}^{t_1} \tilde{k}(t_n-s)  u (x,s) ds  + \sum\limits_{j=2}^{n} \int_{t_{j-1}}^{t_j} \tilde{k}(t_n-s)  u (x,s) ds  \\
     & \approx \omega_{n,1} u^{1} + \sum\limits_{j=2}^{n} \omega_{n,j} u^{j-1/2},
\end{split}
\end{equation}
with $\tilde{k}(t) :=\widetilde{g}'(t)$, $u^{j-1/2}:=\frac{u^{j-1}+u^{j}}{2}$, and the weights
\begin{equation}\label{qwl08}
\begin{split}
     \omega_{n,j} := \int_{t_{j-1}}^{t_j} \tilde{k}(t_n-s)ds = \widetilde{g}(t_n-t_{j-1}) - \widetilde{g}(t_n-t_{j}), \quad  1\leq j \leq n.
\end{split}
\end{equation}
The quadrature errors of the numerical approximation \eqref{qwl07} are given by
\begin{equation}\label{qwl09}
  \begin{split}
     (R_2)^1  & = \int_{0}^{t_1} \tilde{k}(t_1-s)  u (x,s) ds - \omega_{1,1} u^{1}  = \int_{0}^{t_1} \tilde{k}(t_1-s) \Big( \int_{t_1}^{s} \p_{\theta} u(x,\theta)d\theta \Big)  ds, \\
     (R_2)^n  & =  \int_{0}^{t_1} \tilde{k}(t_n-s)  u (x,s) ds - \omega_{n,1} u^{1}  \\
     &\quad +  \sum\limits_{j=2}^{n} \int_{t_{j-1}}^{t_j} \tilde{k}(t_n-s)  u (x,s) ds - \sum\limits_{j=2}^{n} \omega_{n,j} u^{j-1/2} \\
      & = 
      \int_{0}^{t_1} \tilde{k}(t_n-s) \Big( \int_{t_1}^{s} \p_{\theta} u(x,\theta)d\theta \Big)  ds\\
      &\quad+\sum\limits_{j=2}^{n}\int_{t_{j-1}}^{t_j}\tilde{k}(t_n-s) \Big(  \int_{t_{j-1}}^{t_j} K(s;\theta) \p_{\theta}^2 u(x,\theta)d\theta \Big)  ds \\
      &=: (R_{21})^n +(R_{22})^n,
      \quad n \geq 2
  \end{split}
\end{equation}
with the Peano kernel
\begin{equation}\label{ker}
 K(s;\theta) :=
   \begin{cases}
   \displaystyle  \frac{(\theta-t_{j-1})(s-t_j) }{t_j-t_{j-1}}, & t_{j-1}\leq \theta \leq s, \\
   \displaystyle  \frac{(\theta-t_{j})(s-t_{j-1}) }{t_j-t_{j-1}}, & s< \theta \leq t_j.
   \end{cases}
\end{equation}
Substituting \eqref{qwl04_new} and \eqref{qwl07} into \eqref{qwl03_new} yields
\begin{equation}\label{qwl010}
  \begin{split}
      D_N^{\alpha_0} u^n - \Delta u^n = - \Big( \omega_{n,1} u^{1} + \sum\limits_{j=2}^{n} \omega_{n,j} u^{j-1/2} \Big) + F^n + R^n, \quad n = 1,2,\cdots,N
  \end{split}
\end{equation}
with $F^n=F(x,t_n)$ and $R^n=(R_1)^n-(R_2)^n$. Neglecting the truncation error $R^n$ and replacing $u^n$ by its numerical approximation $U^n$, we obtain the following time-semidiscrete scheme
\begin{align}
    & D_N^{\alpha_0} U^n - \Delta U^n = - \Big( \omega_{n,1} U^{1} + \sum\limits_{j=2}^{n} \omega_{n,j} \, U^{j-1/2} \Big) + F^n, \quad n = 1,2,\cdots,N, \label{qwl011} \\
    & U^0 = u_0. \label{qwl012}
\end{align}

\subsection{Stability analysis of time-semidiscrete scheme}

In this subsection, we shall deduce the stability of the time-semidiscrete scheme \eqref{qwl011}--\eqref{qwl012}. Before that, we follow \cite{Liao,Liao1} to introduce the following complementary discrete convolution kernel
\begin{equation}\label{doc}
   P_{n-k}^{(n)} = \frac{1}{a_0^{(k)}}
  \begin{cases}
    1, & k=n, \\
    \sum\limits_{j=k+1}^{n}\left( a_{j-k-1}^{(j)} - a_{j-k}^{(j)}  \right)P_{n-j}^{(n)}, & 1,2,\cdots,n-1.
  \end{cases}
\end{equation}
Then we give several key lemmas of \eqref{doc} for the subsequent analysis.
\begin{lemma}[\cite{Liao,Liao1}]\label{lem01}
Let $P_{n-k}^{(n)}$ denote the complementary discrete convolution kernel. Then the following properties hold:
\begin{itemize}
  \item[(i)] For $1\le k\le n$, it holds that $
    0< P_{n-k}^{(n)} \le \Gamma(2-\alpha_0)\,\tau_k^{\alpha_0}$. 
  \item[(ii)] For $1\le k\le n$, the kernels satisfy the discrete orthogonality relation $
    \sum_{j=k}^{n} P_{n-j}^{(n)}\, a_{j-k}^{(j)} = 1$.
  \item[(iii)] For $q=0,1$ and $1\le n\le N$, one has $
    \sum_{k=1}^{n} P_{n-k}^{(n)}\, \beta_{1+q\alpha_0-\alpha_0}(t_k)
    \leq \beta_{1+q\alpha_0}(t_n)$.
\end{itemize}
\end{lemma}

\begin{lemma}[Discrete fractional Gr\"onwall inequality \cite{Liao}] \label{lem02}
Let  $\{\lambda_l\}_{l=0}^{N-1}$ be a nonnegative sequence satisfying that $ \sum_{l=0}^{N-1}\lambda_l \le \lambda_{*}$, 
where $\lambda_{*}$ is independent of the time steps. Assume that the grid function
$\{v^n\}_{n\ge0}$  satisfies
\begin{equation*}
  D_N^{\alpha_0}(v^n)^2
  \le \sum_{l=1}^{n}\lambda_{n-l}(v^l)^2
  + v^n(\xi^n+\vartheta^n), \quad n \ge 1
\end{equation*}
with nonnegative sequences $\{\xi^n,\vartheta^n\}_{1\le n\le N}$.
If the time grid satisfies $\tau_{n-1}\leq \tau_n$ and $\tau_N \le \sqrt[\alpha_0]{\frac{1}{2\Gamma(2-\alpha_0)\lambda_{*}}}$,
then for $1\le n\le N$, we have
\begin{equation*}
  v^n \le 2E_{\alpha_0}\!\left(2\lambda_{*} t_n^{\alpha_0}\right)
  \left( v^0 + \max_{1\le j\le n}\sum_{l=1}^{j} P_{j-l}^{(j)}\xi^l
     + \beta_{1+\alpha_0}(t_n)\max_{1\le j\le n}\vartheta^j  \right),
\end{equation*}
where the discrete convolution kernel $P_{n-j}^{(n)}$ is defined in \eqref{doc}, and
$
  E_{\alpha}(z)=\sum_{k=0}^{\infty}\frac{z^{k}}{\Gamma(1+k\alpha)}
$
denotes the Mittag-Leffler function.
\end{lemma}

We now establish the stability of the time-semidiscrete scheme
\eqref{qwl011}--\eqref{qwl012}.

\begin{theorem}\label{thm3.3}
Let $\{U^m\}_{m=1}^{N}$ be the solution generated by
\eqref{qwl011}--\eqref{qwl012}. Then the scheme is stable, and the following estimate holds
\begin{equation}\label{thm:eq1}
  \|U^m\| \leq C(T) \Big( \|U^0\| + \max\limits_{1\leq n \leq m}  \sum_{j=1}^{n} P^{(n)}_{n-j} \|F^j\|  \Big), \quad 1\le m\le N.
\end{equation}
\end{theorem}

\begin{proof}
     First, take the inner product of \eqref{qwl011} with $2U^n$ to get
 \begin{align}
     2 \left( D_{N}^{\alpha_0}U^n, U^n \right)
     & + 2  \|\nabla U^n\|^2 =  - 2\Big[ \omega_{n,1} (U^{1},U^n) + \sum\limits_{j=2}^{n} \omega_{n,j} (U^{j-1/2},U^n) \Big] + 2 (F^n,U^n). \label{zxc05}
\end{align}
Multiplying \eqref{zxc05} by $P^{(m)}_{m-n}$ and summing over $n=1,\ldots,m$ yields
 \begin{align}
     2 \sum_{n=1}^{m} &  P^{(m)}_{m-n}\left( D_{N}^{\alpha_0}U^n, U^n \right)
      + 2 \sum_{n=1}^{m} P^{(m)}_{m-n} \|\nabla U^n\|^2\nonumber\\
      & = - 2 \sum_{n=1}^{m} P^{(m)}_{m-n}\Big[ \omega_{n,1} (U^{1},U^n) + \sum\limits_{j=2}^{n} \omega_{n,j} (U^{j-1/2},U^n) \Big]  + 2 \sum_{n=1}^{m} P^{(m)}_{m-n} (F^n,U^n)  \label{ww01} \\
      & =: \Lambda_1 + \Lambda_2. \nonumber
\end{align}
We next examine each term in \eqref{ww01}. In particular, \cite[Theorem~2.1]{Liao} implies that
 \begin{align}\label{mm01}
     2  \left( D_{N}^{\alpha_0}U^n, U^n \right)
     \geq  \sum_{k=1}^{n}a^{(n)}_{n-k} \delta_{\tau} (\|U^k\|^2) = D_{N}^{\alpha_0} (\|U^n\|)^2.
\end{align}
which follows Lemma \ref{lem01} (ii) to obtain
 \begin{align}
     \sum_{n=1}^{m} P^{(m)}_{m-n}  \sum_{k=1}^{n}a^{(n)}_{n-k} \delta_{\tau} (\|U^k\|^2) & =  \sum_{k=1}^{m} \delta_{\tau}(\|U^k\|^2)  \sum_{n=k}^{m} P^{(m)}_{m-n} a^{(n)}_{n-k} \nonumber \\
     & = \sum_{k=1}^{m} \delta_{\tau}(\|U^k\|^2) = \|U^m\|^2 - \|U^0\|^2. \label{ww02}
\end{align}
Then we analyze $\Lambda_1$. We utilize the Poincar\'{e} inequality ($\| U\|\leq c_{0}\|\nabla U\|$), Young inequality
and
\begin{align}
    |\omega_{n,n}| &\leq \int_{t_{n-1}}^{t_{n}}|\tilde{k}(t_n-s)|ds \leq C\int_{t_{n-1}}^{t_{n}}(t_n-s)^{-\alpha^*}ds \leq c_0'\tau_n^{1-\alpha^*}, \quad \text{(see \eqref{qq1}, \eqref{qwl08})} \nonumber \\
    \sum\limits_{j=1}^{n-1} & (|\omega_{n,j}|+|\omega_{n,j+1}|)  \leq 2  \sum\limits_{j=1}^{n}|\omega_{n,j}| \leq 2C\int_{0}^{t_{n}}(t_n-s)^{-\alpha^*}ds \leq c_{0}'  \label{mm02}
\end{align}
to yield
\begin{align}
  \Phi^n:= & \Big| -2\Big[ \omega_{n,1}  (U^{1},U^n) + \sum\limits_{j=2}^{n} \omega_{n,j} (U^{j-1/2},U^n)  \Big] \Big| \nonumber \\
  & \leq 2\sum\limits_{j=1}^{n} |\omega_{n,j}| \|U^{j}\| \|U^n\| + \sum\limits_{j=2}^{n} |\omega_{n,j}| \|U^{j-1}\| \|U^n\| \nonumber \\
& \leq  2c_0^2\sum\limits_{j=1}^{n} |\omega_{n,j}| \|\nabla U^{j}\| \|\nabla U^n\| + 2 c_0^2 \sum\limits_{j=2}^{n} |\omega_{n,j}| \|\nabla U^{j-1}\| \|\nabla U^n\| \nonumber \\
&  = 2c_0^2|\omega_{n,n}| \|\nabla U^n\|^2 + 2c_0^2\sum\limits_{j=1}^{n-1} (|\omega_{n,j}|+|\omega_{n,j+1}|) \|\nabla U^{j}\| \|\nabla U^n\| \nonumber \\
& \leq 2c_0^2|\omega_{n,n}| \|\nabla U^n\|^2 + 2c_0^2\sum\limits_{j=1}^{n-1} (|\omega_{n,j}|+|\omega_{n,j+1}|) \big( C_{\epsilon}\|\nabla U^{j}\|^2 + \epsilon  \|\nabla U^n\|^2\big) \nonumber \\
& \leq \frac{1}{2}\|\nabla U^n\|^2 + 2c_0^2C_{\epsilon} \sum\limits_{j=1}^{n-1} (|\omega_{n,j}|+|\omega_{n,j+1}|) \|\nabla U^{j}\|^2, \label{mod1}
\end{align}
where we select $\tau_n^{1-\alpha^*}+\epsilon \leq \frac{1}{2c_0^2c_0'}$. Hence, we obtain with $c_0'':=2c_0^2 C_{\epsilon} $
\begin{align}
  |\Lambda_1|  \leq  \sum_{n=1}^{m} P^{(m)}_{m-n}  \|\nabla U^{n}\|^2  & + c_0'' \sum_{n=1}^{m} P^{(m)}_{m-n}  \sum\limits_{j=1}^{n-1} (|\omega_{n,j}|+|\omega_{n,j+1}|) \|\nabla U^{j}\|^2 \nonumber \\
&  =: \sum_{n=1}^{m} P^{(m)}_{m-n}  \|\nabla U^{n}\|^2  + \Xi. \label{qq2}
\end{align}
Then for $\Xi$, we first use \cite[Equation (4.6)]{ZheQiu} to get
\begin{align}
   |\omega_{n,j+1}|  \leq C\int_{t_{j}}^{t_{j+1}}(t_n-s)^{-\alpha^*} ds & \leq C\Big[ (t_n-t_j)^{1-\alpha^*} - (t_n-t_{j+1})^{1-\alpha^*} \Big] \nonumber \\
& \leq c_0'''\frac{N^{-(1-\alpha^*)}}{[n-(j+1)]^{\alpha^*}} \leq c_0''' 2^{\alpha^*} \frac{N^{-(1-\alpha^*)}}{(n-j)^{\alpha^*}}, \label{mm03}
\end{align}
which implies with $c_0^{(4)}:=(1+2^{\alpha^*}) c_0''c_0'''$
\begin{align*}
    \Xi & \leq c_0^{(4)} \sum_{n=1}^{m} P^{(m)}_{m-n}  \sum\limits_{j=1}^{n-1}  \frac{N^{-(1-\alpha^*)}}{(n-j)^{\alpha^*}}  \|\nabla U^{j}\|^2 =  c_0^{(4)} \sum_{n=1}^{m} P^{(m)}_{m-n}  \sum\limits_{j=1}^{n-1}  \frac{N^{-(1-\alpha^*)}}{j^{\alpha^*}}  \|\nabla U^{n-j}\|^2.
\end{align*}
We then interchange the order of summation to obtain
\begin{equation}\label{ww05}
\begin{split}
        \Xi & \leq  c_0^{(4)} \sum\limits_{j=1}^{m-1}  \frac{N^{-(1-\alpha^*)}}{j^{\alpha^*}}  \sum_{n=j+1}^{m} P^{(m)}_{m-n}  \|\nabla U^{n-j}\|^2  \\
            &  \overset{\ell = n-j}{=}   c_0^{(4)} \sum\limits_{j=1}^{m-1}  \frac{N^{-(1-\alpha^*)}}{j^{\alpha^*}}   \sum_{\ell = 1}^{m-j} P^{(m-j)}_{m-j-\ell}  \|\nabla U^{\ell}\|^2 \\
&  \overset{n=m-j}{=}     c_0^{(4)} \sum\limits_{n=1}^{m-1}  \frac{N^{-(1-\alpha^*)}}{(m-n)^{\alpha^*}} \sum_{\ell = 1}^{n} P^{(n)}_{n-\ell}  \|\nabla U^{\ell}\|^2.
\end{split}
\end{equation}
Next, we analyze $\Lambda_2$. Apply Young's inequality to obtain
\begin{align}
  | \Lambda_2| & \leq 2c_0 \sum_{n=1}^{m} P^{(m)}_{m-n} \|F^n\| \| U^n\|  \label{ww010}.
\end{align}
We then put \eqref{qq2}, \eqref{ww05} and \eqref{ww010} into \eqref{ww01}, use \eqref{ww02}, and eliminate the term $\sum_{n=1}^{m} P^{(m)}_{m-n}  \|\nabla U^{n}\|^2$ in \eqref{qq2} to get
\begin{align*}
    \|U^m\|^2 & + \sum_{n=1}^{m} P^{(m)}_{m-n}  \|\nabla U^{n}\|^2  \\
    & \quad  \leq \|U^0\|^2 + c_0^{(4)} \sum\limits_{n=1}^{m-1}  \frac{N^{-(1-\alpha^*)}}{(m-n)^{\alpha^*}} \sum_{\ell = 1}^{n} P^{(n)}_{n-\ell}  \|\nabla U^{\ell}\|^2 + 2c_0 \sum_{n=1}^{m} P^{(m)}_{m-n} \|F^n\| \| U^n\|,
\end{align*}
which follows the notations
\begin{align*}
    \|U^n\|_A := \sqrt{\sum_{\ell = 1}^{n} P^{(n)}_{n-\ell}  \|\nabla U^{\ell}\|^2}, \quad  \|U^n\|_B := \sqrt{\|U^n\|^2 + \|U^n\|_A^2}, \quad 1\leq n \leq m
\end{align*}
to arrive at
\begin{align*}
    \|U^m\|_B^2 &  \leq \|U^0\|^2  + c_0^{(4)} \sum\limits_{n=1}^{m-1}  \frac{N^{-(1-\alpha^*)}}{(m-n)^{\alpha^*}} \|U^n\|_B^2 + 2c_0 \sum_{n=1}^{m} P^{(m)}_{m-n} \|F^n\| \|U^n\| \\
&  \leq \Big[ \|U^0\|^2 + 2c_0 \max\limits_{1\leq n \leq m}\|U^n\|  \sum_{n=1}^{m} P^{(m)}_{m-n} \|F^n\| \Big]  + c_0^{(4)} \sum\limits_{n=1}^{m-1}  \frac{N^{-(1-\alpha^*)}}{(m-n)^{\alpha^*}} \|U^n\|_B^2 \\
& \leq \Big[ \|U^0\|^2 + 2c_0 \max\limits_{1\leq n \leq m}  \sum_{j=1}^{n} P^{(n)}_{n-j} \|F^j\|   \max\limits_{1\leq n \leq m}\|U^n\| \Big]  + c_0^{(4)} \sum\limits_{n=1}^{m-1}  \frac{N^{-(1-\alpha^*)}}{(m-n)^{\alpha^*}} \|U^n\|_B^2.
\end{align*}
We then use the discrete Gr\"{o}nwall lemma of convolution type \cite[Theorem 6.1.19]{Bru} to arrive at
\begin{align*}
    \|U^m\|_B^2 \leq E_{1-\alpha^*}\Big( c_0^{(4)} \Gamma(1-\alpha^*) \big(\frac{m}{N}\big)^{1-\alpha^*} \Big)  \Big[ \|U^0\|^2 + 2c_0  \max\limits_{1\leq n \leq m}  \sum_{j=1}^{n} P^{(n)}_{n-j} \|F^j\|   \max\limits_{1\leq n \leq m}\|U^n\| \Big],
\end{align*}
which follows $\|U^m\|_B\geq \|U^m\|$ to obtain
\begin{align*}
    \|U^m\|^2 & \leq C(T) \Big[ \|U^0\|^2 +  \max\limits_{1\leq n \leq m}  \sum_{j=1}^{n} P^{(n)}_{n-j} \|F^j\|  \max\limits_{1\leq n \leq m}\|U^n\| \Big].
\end{align*}
Take a suitable $M$ such that $\|U^M\| =\max\limits_{0\leq n \leq m}\|U^n\|$, thus we obtain
\begin{align*}
    \|U^M\|^2 &  \leq C(T) \Big[ \|U^0\|\left\|U^M \right\| + \max\limits_{1\leq n \leq M}  \sum_{j=1}^{n} P^{(n)}_{n-j} \|F^j\|  \max\limits_{1\leq n \leq M}\|U^n\| \Big] \\
    &  \leq C(T) \Big[ \|U^0\|\left\|U^M \right\| + \max\limits_{1\leq n \leq m}  \sum_{j=1}^{n} P^{(n)}_{n-j} \|F^j\|  \max\limits_{0\leq n \leq m}\|U^n\| \Big] \\
    &  = C(T) \Big[ \|U^0\|\left\|U^M \right\| + \max\limits_{1\leq n \leq m}  \sum_{j=1}^{n} P^{(n)}_{n-j} \|F^j\|  \left\|U^M \right\| \Big],
\end{align*}
which completes the proof by using $\|U^M\|\geq \|U^m\|$.
\end{proof}

\begin{remark}\label{rm2.4}

To control the first term on the right-hand side of \eqref{zxc05} by $2\|\nabla U^n\|$,  we apply Poincar\'{e} inequality in \eqref{mod1} and subsequently apply the discrete Gr\"{o}nwall lemma of convolution type (cf. \cite[Theorem~6.1.19]{Bru}) to carry out analysis, which requires delicate and complicated treatment.
In fact, this proof could be further simplified by incorporating the discrete fractional Gr\"onwall inequality (cf. Lemma \ref{lem02}) to deal with \eqref{zxc05}. Specifically, we first modify the proof of \eqref{mod1} to obtain
\begin{align}
 \Phi^n
& \leq 2c_0^2 |\omega_{n,n}| \|\nabla U^n\|^2 + 2 \Big[ \epsilon_0\| U^n\|^2 c_0'  +     \frac{1}{4 \epsilon_0}\sum\limits_{j=1}^{n-1} (|\omega_{n,j}|+|\omega_{n,j+1}|) \| U^{j}\|^2 \Big] \nonumber \\
& \leq 2c_0^2c_0' \left(\tau_n^{1-\alpha^*} + \epsilon_0 \right) \|\nabla U^n\|^2 +  \frac{1}{2 \epsilon_0}\sum\limits_{j=1}^{n-1} (|\omega_{n,j}|+|\omega_{n,j+1}|) \| U^{j}\|^2 \nonumber \\
& \leq 2 \|\nabla U^n\|^2 +  \frac{c_0^2c_0'}{2}\sum\limits_{j=1}^{n-1} (|\omega_{n,j}|+|\omega_{n,j+1}|) \| U^{j}\|^2, \label{phqq}
\end{align}
where we used \eqref{mm02} and take $\tau_n^{1-\alpha^*}\leq \frac{1}{c_0^2c_0'}$ with $\epsilon_0=\frac{1}{c_0^2c_0'}$. Thus, following \eqref{zxc05} and \eqref{mm01}, we apply \eqref{phqq}, \eqref{mm03} and denote $\lambda_j=\frac{c_0^2c_0'c_0''' 2^{\alpha^*}}{N^{(1-\alpha^*)}} j^{-\alpha^*}$ to yield
 \begin{align}
  \hspace{-0.15in}   D_N^{\alpha_0}(\|U^n\|)^2 &  \hspace{-0.05in} \leq \frac{c_0^2c_0'}{2} \sum\limits_{j=1}^{n-1} c_0''' 2^{1+\alpha^*} \frac{N^{-(1-\alpha^*)}}{(n-j)^{\alpha^*}} \| U^{j}\|^2 \!+\! 2 \|F^n\| \|U^n\| =  \sum\limits_{j=1}^{n-1} \lambda_{n-j} \| U^{j}\|^2 + 2 \|F^n\| \|U^n\|, \label{qqww01}
\end{align}
which uses Lemma \ref{lem02} (taking $\xi^n =2 \|F^n\|$ and $\vartheta^n =0$) with
\begin{align} \label{qqww02}
    \sum\limits_{j=1}^{n-1} \lambda_{j} = \frac{c_0^2c_0'c_0''' 2^{\alpha^*}}{N^{(1-\alpha^*)}} \sum\limits_{j=1}^{n-1}  j^{-\alpha^*} \leq \frac{c_0^2c_0'c_0''' 2^{\alpha^*}}{N^{(1-\alpha^*)}}\frac{n^{1-\alpha^*}}{1-\alpha^*} \leq \frac{c_0^2c_0'c_0''' 2^{\alpha^*}}{1-\alpha^*} =: \lambda_*
\end{align}
to obtain
$  \|U^n\|
  \leq 2E_{\alpha_0}\!\left(2\lambda_{*} t_n^{\alpha_0}\right)
  \Big( \|U^0\| + 2\max\limits_{1\le j\le n}\sum\limits_{l=1}^{j} P_{j-l}^{(j)} \|F^l\| \Big).
  $
This gives \eqref{thm:eq1}.
\end{remark}

\begin{remark} 
     We note that the stability analysis framework of Theorem \ref{thm3.3} also applies to more general problems, e.g., the following extended model with non-positive memory \cite{Qiu}
   \begin{align}\label{ss}
        \partial_t^{\alpha_0}u(x,t)-\Delta u(x,t)  = -(\mathcal{K} * \mathcal{A} u)(x,t) + f(x,t), \quad 0<t\leq T, \quad 0<\alpha_0 <1,
   \end{align}
where the kernel satisfies $|\mathcal{K}(t)|\leq C_0 \, t^{-\mu}$ with $0<\mu<1$ and the operator $\mathcal{A} := \sum_{j=1}^{2} \tilde{w}_j \Delta^{j-1}$ with $|\tilde{w}_j |\leq 1$. Without loss of generality, we still denote the solution of \eqref{ss} by $u$. We approximate the first and third terms of model \eqref{ss} by the nonuniform L1 formula \eqref{qwl04_new} and interpolation quadrature rule \eqref{qwl07} given in Section \ref{sec3}, which gives the following time-semidiscrete scheme
 \begin{align}
     & D_N^{\alpha_0} U^n - \Delta U^n = -I_N^{\tilde{\omega}}\mathcal{A}U^n + f^n, \quad 1\leq n \leq N, \quad U^0 = u_0. \label{q004}
 \end{align}
Here $I_N^{\tilde{\omega}}\mathcal{A}U^n:=\big( \tilde{\omega}_{n,1} \mathcal{A}U^{1} + \sum_{j=2}^{n} \tilde{\omega}_{n,j} \mathcal{A}U^{j-1/2} \big)$ and $\tilde{\omega}_{n,j}:=\int_{t_{j-1}}^{t_j} \mathcal{K}(t_n-s)ds$ such that
 \begin{equation}\label{q002}
 \begin{split}
      |\tilde{\omega}_{n,j}|  \leq C\big[ (t_n-t_{j-1})^{1-\mu} - (t_n-t_{j})^{1-\mu} \big]\leq C \times 
      \begin{cases}
        \displaystyle  \frac{N^{-(1-\mu)}}{(n-j)^{\mu}}, & 1\leq j \leq n-1, \\[0.1in]
        \displaystyle  N^{-(1-\mu)}, & j = n.
      \end{cases}
 \end{split}
 \end{equation}
For the scheme \eqref{q004}, we use \eqref{q002} and follow the argument of Theorem \ref{thm3.3} (where we replace $\alpha^*$ with $\mu$) to establish the following stability result
 \begin{align}
   \|U^m\| \leq C(T) \Big( \|U^0\| + \max\limits_{1\leq n \leq m}  \sum_{j=1}^{n} P^{(n)}_{n-j} \|f^j\|  \Big), \quad 1\le m\le N. \label{stabi}
 \end{align}
    Based on the regularity result \cite[Equation~(11)]{Qiu}, we establish the following convergence result for the scheme \eqref{q004}
\begin{align}\label{conv3}
        \|u^m-U^m\| \leq C(T) N^{-\min\{2-\alpha_0, \, r\alpha_0 \}}, \quad 1\leq m \leq N.
\end{align}

In summary, the above result \eqref{conv3} improves the first-order temporal convergence rate proved in \cite{Qiu} to the order $2-\alpha_0$ provided that $r \ge \frac{2-\alpha_0}{\alpha_0}$, and eliminates the restriction $0<\mu<\min\{1,\,2\alpha_0\}$ imposed in the non-positive memory case; see \cite[Equation~(4)]{Qiu}. The detailed proof of \eqref{conv3} is provided in the Appendix.
\end{remark}

\section{Error  analysis}\label{sec4}
\subsection{Error estimate of time-semidiscrete scheme}

We deduce the convergence of the time-semidiscrete scheme as well as the fully discrete scheme.

\begin{lemma}\cite{Qiu} \label{lem03}
  Based on the regularity assumption \eqref{regu} and Lemma \ref{lem01}, we yield
   \begin{align*}
       \sum_{n=1}^{m}  P_{m-n}^{(m)} \left\|(R_1)^n\right\| \leq C(T) N^{-\min\{2-\alpha_0, \, r\alpha_0\}}, \quad 1\leq n \leq  m \leq N.
   \end{align*}
\end{lemma}

Then we analyze $(R_2)^n$ in \eqref{qwl09} with $n\geq 1$. We give the following lemma.
\begin{lemma} \label{lem04}
  Under the regularity assumption \eqref{regu}, it holds that
   \begin{align*}
       &\left\|(R_2)^1\right\| \leq Ct_1 (1+|\ln(t_1)| ),  \\
       &\left\|(R_2)^n\right\| \leq Ct_{n}^{-\alpha_0} \left(1+ \ln N \right) \Big[ N^{-\min\{ 2,  \,  r(1+\alpha_0) \}} + \Theta^N(\alpha_0,r) \Big], \quad 2\leq n \leq N,
   \end{align*}
   where $\Theta^N(\alpha_0,r)$ is defined as follows
   \begin{equation}\label{theta}
   \Theta^N(\alpha_0,r):=
    \begin{cases}
        N^{-2},  &  r(1+\alpha_0)>2, \nonumber \\
       N^{-2} \ln N, &  r(1+\alpha_0)=2, \nonumber \\
       N^{-r(1+\alpha_0)}, &  r(1+\alpha_0)<2.
     \end{cases} \nonumber \\
   \end{equation}
\end{lemma}
\begin{proof}
  For  $(R_2)^n$ in \eqref{qwl09}, we will use the regularity assumption \eqref{regu} to derive the desired result. \par
  \textbf{I: Estimate of $\|(R_2)^1\|$.} 
First, we employ \eqref{regu} to obtain
\begin{align*}
   \|(R_2)^1\| & \leq  \int_{0}^{t_1} \left|\tilde{k}(t_1-s)\right|  \Big( \int_{0}^{t_1} \|\p_{\theta} u(\cdot,\theta) \| d\theta \Big) ds \leq C  \int_{0}^{t_1} |\ln(t_1-s)| (t_1-s)^{-\alpha_0}  t_1^{\alpha_0} ds \\[0.1in]
& = C t_1^{\alpha_0} \int_{0}^{t_1} \frac{-\ln(t_1-s)}{(t_1-s)^{\alpha_0}} ds \leq C t_1^{\alpha_0} \int_{0}^{t_1} \ln(t_1-s) d\big( (t_1-s)^{1-\alpha_0} \big) \\[0.1in]
& =  C t_1^{\alpha_0} \Big[ -\ln(t_1) t_1^{1-\alpha_0} + \int_{0}^{t_1}(t_1-s)^{-\alpha_0} ds  \Big] \leq Ct_1 (1+|\ln(t_1)| ).
\end{align*}

\textbf{II: Estimate of $\|(R_{21})^n\|$ with $n\geq 2$.}
First, we have $t_n-t_1\geq t_1$, satisfying for $s\in (0,t_1)$
  \begin{align*}
     |\tilde{k}(t_n-s)| & \leq C(t_n-s)^{-\alpha_0} (1+|\ln(t_n-s)|)  \\[0.1in]
     &\leq C(t_n-t_1)^{-\alpha_0} (1+|\ln(t_n-t_1)|) \leq C t_n^{-\alpha_0} (1+|\ln(t_1)|),
  \end{align*}
thus we use \eqref{regu} to get
\begin{align*}
   \|(R_{21})^n\| & \leq  \int_{0}^{t_1} |\tilde{k}(t_n-s)|  \Big( \int_{s}^{t_1} \|\p_{\theta} u(\cdot,\theta) \| d\theta \Big) ds \leq C t_n^{-\alpha_0} t_1  (1+|\ln(t_1)|) \int_{0}^{t_1}  \theta^{\alpha_0-1} d\theta \\[0.1in]
&    \leq Ct_n^{-\alpha_0} t_1^{\alpha_0+1} (1+\ln N) \leq Ct_n^{-\alpha_0} N^{-r(1+\alpha_0)} (1+\ln N).
\end{align*}

\textbf{III: Estimate of $\|(R_{22})^n\|$ with $n\geq 2$.} We rewrite $(R_{22})^n$ as
\begin{align*}
       (R_{22})^n  & = \int_{t_{n-1}}^{t_n}\tilde{k}(t_n-s) \Big(  \int_{t_{n-1}}^{t_n} K(s;\theta) \p_{\theta}^2 u(x,\theta)d\theta \Big)  ds \\[0.1in]
      & \quad + \sum\limits_{j=2}^{n-1}\int_{t_{j-1}}^{t_j}\tilde{k}(t_n-s) \Big(  \int_{t_{j-1}}^{t_j} K(s;\theta) \p_{\theta}^2 u(x,\theta)d\theta \Big)  ds =: (R_{221})^n + (R_{222})^n.
\end{align*}
We first estimate $(R_{221})^n$ by using \eqref{regu} and \eqref{ker}
\begin{align*}
   \|(R_{221})^n\| & \leq C\tau_n \int_{t_{n-1}}^{t_n} |\tilde{k}(t_n-s)| \Big( \int_{t_{n-1}}^{t_n} \|\p_{\theta}^2 u(\cdot,\theta)\|d\theta  \Big) ds \\
   & \leq C\tau_n^2 t_{n-1}^{\alpha_0-2} \int_{t_{n-1}}^{t_n} \big(1+|\ln(t_n-s)| \big) (t_n-s)^{-\alpha_0} ds,
\end{align*}
which follows
\begin{align*}
     \int_{t_{n-1}}^{t_n}  |\ln(t_n-s)| &  (t_n-s)^{-\alpha_0} ds = \int_{t_{n-1}}^{t_n} -\ln(t_n-s)  (t_n-s)^{-\alpha_0} ds \\
     & = \int_{t_{n-1}}^{t_n} \ln(t_n-s)  (t_n-s)^{-\alpha_0} d\Big( \frac{(t_n-s)^{1-\alpha_0} }{1-\alpha_0} \Big) \\
     & = - \ln(\tau_n)  \tau_n^{1-\alpha_0} + \int_{t_{n-1}}^{t_n} \frac{(t_n-s)^{1-\alpha_0} }{-\alpha_0} ds \\
     & \leq C \tau_n^{1-\alpha_0} (1+|\ln (\tau_n)|)
\end{align*}
to yield with $n\geq 2$
\begin{align*}
   \|(R_{221})^n\| &   \leq C\tau_n^{3-\alpha_0} t_{n}^{2\alpha_0-2} t_{n}^{-\alpha_0} (1+|\ln (\tau_n)|) \\[0.1in]
   & \leq C t_{n}^{-\alpha_0} (1+|\ln (\tau_1)|) \frac{n^{(r-1)(3-\alpha_0)}}{N^{r(3-\alpha_0)}} \Big(\frac{n}{N}\Big)^{2r(\alpha_0-1)} \\[0.1in]
& \leq Ct_{n}^{-\alpha_0} (1+ r|\ln \tau|) \frac{n^{r(1+\alpha_0)-(3-\alpha_0)} }{N^{r(1+\alpha_0)-(3-\alpha_0)}} N^{-(3-\alpha_0)} \\[0.1in]
&  \leq Ct_{n}^{-\alpha_0} (1+ \ln N)\, N^{-\min\{ 3-\alpha_0,  \, r(1+\alpha_0) \}}.
\end{align*}
We then estimate $(R_{222})^n$ and rewrite it as
\begin{align*}
   (R_{222})^n & = \sum\limits_{j=\lceil n/2 \rceil+1}^{n-1}\int_{t_{j-1}}^{t_j}\tilde{k}(t_n-s) \Big(  \int_{t_{j-1}}^{t_j} K(s;\theta) \p_{\theta}^2 u(x,\theta)d\theta \Big)  ds \\[0.1in]
 & \quad + \sum\limits_{j=2}^{\lceil n/2 \rceil} \int_{t_{j-1}}^{t_j}\tilde{k}(t_n-s) \Big(  \int_{t_{j-1}}^{t_j} K(s;\theta) \p_{\theta}^2 u(x,\theta)d\theta \Big)  ds =: (R_{2221})^n + (R_{2222})^n,
\end{align*}
where $\lceil \cdot \rceil$ denotes the ceiling operator. We see that, for $2\le j \leq n-1$,
\begin{align*}
   \Big\| \int_{t_{j-1}}^{t_j}\tilde{k}(t_n-s) & \Big(  \int_{t_{j-1}}^{t_j} K(s;\theta) \p_{\theta}^2 u(\cdot,\theta) d\theta \Big)  ds  \Big\| \\[0.1in]
   & \leq C\tau_j \int_{t_{j-1}}^{t_j}|\tilde{k}(t_n-s)| \int_{t_{j-1}}^{t_j}\|\p_{\theta}^2 u(\cdot,\theta)\| d\theta  ds \\[0.1in]
   & \leq C\tau_j \int_{t_{j-1}}^{t_j} (t_n-s)^{-\alpha_0} (1+|\ln(t_n-s)|) \int_{t_{j-1}}^{t_j}\theta^{\alpha_0-2} d\theta  ds \\[0.1in]
   & \leq C\tau_j^2 t_j^{\alpha_0-2} \int_{t_{j-1}}^{t_j} (t_n-s)^{-\alpha_0} (1+|\ln(t_n-s)|)   ds \\[0.1in]
   & \leq C\tau_j^2 t_j^{\alpha_0-2} (1+|\ln(\tau_n)|) \int_{t_{j-1}}^{t_j} (t_n-s)^{-\alpha_0} ds.
\end{align*}
Thus, for $\lceil n/2 \rceil +1 \leq j \leq n-1$, we have $t_j\geq 2^{-r}t_n$, which implies
\begin{align*}
    \|(R_{2221})^n\| & \leq  \sum\limits_{j=\lceil n/2 \rceil+1}^{n-1}  \int_{t_{j-1}}^{t_j} |\tilde{k}(t_n-s)| \Big(  \int_{t_{j-1}}^{t_j} |K(s;\theta)|  \left\| \p_{\theta}^2 u(\cdot,\theta) \right\| d\theta  \Big)  ds  \\[0.1in]
  & \leq  C\sum\limits_{j=\lceil n/2 \rceil+1}^{n-1}  \int_{t_{j-1}}^{t_j} |\tilde{k}(t_n-s)|   \tau_j^2 t_{j-1}^{\alpha_0-2} ds \\[0.1in]
  & \leq  C \tau_n^2 \, (t_{\lceil n/2 \rceil})^{\alpha_0-2}  \sum\limits_{j=\lceil n/2 \rceil+1}^{n-1}  \int_{t_{j-1}}^{t_j} |\tilde{k}(t_n-s)|   ds \\[0.1in]
  & \leq  C \tau_n^2 \, 2^{r(2-\alpha_0)}  t_{n}^{\alpha_0-2} \int_{t_{\lceil n/2 \rceil}}^{t_{n-1}} |\tilde{k}(t_n-s)|  ds.
\end{align*}
Then, we further obtain
\begin{align*}
   \|(R_{2221})^n\|
  & \leq  C \tau_n^2 \,  t_{n}^{2\alpha_0-2}  t_{n}^{-\alpha_0} \left(1+|\ln (\tau_n)|\right) t_n^{1-\alpha_0}  =  C \tau_n^2 \,  t_{n}^{\alpha_0-1}  t_{n}^{-\alpha_0} \left(1+|\ln (\tau_n)|\right)  \\[0.1in]
  & \leq Ct_{n}^{-\alpha_0} (1+ r|\ln \tau|) \frac{n^{r(1+\alpha_0)-2} }{N^{r(1+\alpha_0)-2}} N^{-2}  \leq Ct_{n}^{-\alpha_0} \left(1+ \ln N \right)  N^{-\min\{ 2, \,  r(1+\alpha_0) \}}.
\end{align*}
Subsequently, for $ 2 \leq j \leq \lceil n/2 \rceil $, we get $(t_n-t_j)^{-\alpha_0} \leq  (1-2^{-r})^{-\alpha_0} t_n^{-\alpha_0} $, which leads to
\begin{align*}
    \|(R_{2222})^n\| & \leq  \sum\limits_{j=2}^{\lceil n/2 \rceil}  \int_{t_{j-1}}^{t_j} |\tilde{k}(t_n-s)| \Big(  \int_{t_{j-1}}^{t_j} |K(s;\theta)|  \left\| \p_{\theta}^2 u(\cdot,\theta) \right\| d\theta  \Big)  ds  \\
  & \leq  C\sum\limits_{j=2}^{\lceil n/2 \rceil}  \int_{t_{j-1}}^{t_j} |\tilde{k}(t_n-s)|   \tau_j^2 t_{j-1}^{\alpha_0-2} ds  \leq  C(1+|\ln (\tau_n)|) \sum\limits_{j=2}^{\lceil n/2 \rceil}  \tau_j (t_n-t_j)^{-\alpha_0}  \tau_j^2 t_{j}^{\alpha_0-2}  \\
& \leq C (1+|\ln (\tau_n)|) t_n^{-\alpha_0} \sum\limits_{j=2}^{\lceil n/2 \rceil}  t_{j}^{\alpha_0-2}  \tau_j^3
\leq C (1+|\ln \tau|) t_n^{-\alpha_0} \sum\limits_{j=2}^{\lceil n/2 \rceil} \frac{j^{-r(2-\alpha_0)}}{N^{-r(2-\alpha_0)}} \frac{j^{3(r-1)}}{N^{3r}},
\end{align*}
which, together with \eqref{theta}, further gives
\begin{align}
    \|(R_{2222})^n\| & \leq  C (1+|\ln \tau|) t_n^{-\alpha_0} N^{-r(1+\alpha_0)} \sum\limits_{j=2}^{\lceil n/2 \rceil} j^{r(1+\alpha_0)-3} \nonumber \\
    &  \leq  C (1+\ln N) \, t_n^{-\alpha_0} \times
     \begin{cases}
        N^{-2},  &  r(1+\alpha_0)>2, \nonumber \\
       N^{-2} \ln N, &  r(1+\alpha_0)=2, \nonumber \\
       N^{-r(1+\alpha_0)}, &  r(1+\alpha_0)<2
     \end{cases} \nonumber \\
     & = C (1+\ln N)\, t_n^{-\alpha_0} \, \Theta^N(\alpha_0,r), 
\end{align}
where we used the fact that
\begin{align}\label{bs}
    \sum\limits_{j=2}^{\lceil n/2 \rceil} j^{r(1+\alpha_0)-3}  &\leq  C \times
     \begin{cases}
       n^{r(1+\alpha_0) -2},  &  r(1+\alpha_0)>2,   \\
       \ln n, &  r(1+\alpha_0)=2,   \\
       1, &  r(1+\alpha_0)<2
     \end{cases} \\
     & \leq C \times
     \begin{cases}
       N^{r(1+\alpha_0) -2},  &  r(1+\alpha_0)>2,   \\
       \ln N, &  r(1+\alpha_0)=2,    \\
       1, &  r(1+\alpha_0)<2 \\
     \end{cases} \nonumber\\
    & = C\Theta^N(\alpha_0,r).\nonumber
\end{align}
We then complete the proof of this lemma.
\end{proof}

Thus, we further obtain the following result.

\begin{lemma} \label{lem05}
  Under Lemma \ref{lem04} and Lemma \ref{lem01}, we have
   \begin{align*}
       \sum_{n=1}^{m}  P_{m-n}^{(m)} \left\|(R_2)^n\right\| \leq C (1+ \ln N )^2 N^{-\min\{2, \, r(1+\alpha_0) \}}, \quad 1\leq n \leq  m \leq N.
   \end{align*}
\end{lemma}

\begin{proof} First, we rewrite $\sum_{n=1}^{m}  P_{m-n}^{(m)} \left\|(R_2)^n\right\|$ as
\begin{align*}
   \sum_{n=1}^{m}  P_{m-n}^{(m)} \left\|(R_2)^n\right\| = P_{m-1}^{(m)} \left\|(R_2)^1\right\| + \sum_{n=2}^{m}  P_{m-n}^{(m)} \left\|(R_2)^n\right\|.
\end{align*}
We use Lemma \ref{lem01} (i) and Lemma \ref{lem04} to obtain
\begin{align*}
    P_{m-1}^{(m)} \left\|(R_2)^1\right\| \leq C\tau_1^{\alpha_0} t_1 ( 1+|\ln(t_1)| ) \leq C\tau_1^{1+\alpha_0} ( 1+\ln N ) \leq CN^{-r(1+\alpha_0)} ( 1+\ln N ).
\end{align*}
Similarly, we get
\begin{align*}
    \sum_{n=2}^{m}  P_{m-n}^{(m)} \left\|(R_2)^n\right\| \leq C\left(1+ \ln N \right)  \sum_{n=2}^{m}  P_{m-n}^{(m)} t_{n}^{-\alpha_0} \Big[ N^{-\min\{ 2, \,  r(1+\alpha_0) \}}  + \Theta^N(\alpha_0,r) \Big],
\end{align*}
in which we follow Lemma \ref{lem01} (iii) to yield
\begin{align}
  \sum_{n=2}^{m}  P_{m-n}^{(m)} t_{n}^{-\alpha_0} \leq  \Gamma(1-\alpha_0) \sum_{n=1}^{m}   P_{m-n}^{(m)} \beta_{1-\alpha_0}(t_n) \leq  \Gamma(1-\alpha_0) \beta_{1}(t_m), \label{bbb}
\end{align}
which implies that
\begin{align*}
  \sum_{n=2}^{m}  P_{m-n}^{(m)} \left\|(R_2)^n\right\| & \leq  C\left(1+ \ln N \right) \Big[ N^{-\min\{ 2, \,  r(1+\alpha_0) \}}  +  N^{-\min\{ 2,  \, r(1+\alpha_0) \}} \ln N \Big] \\
& \leq C\left(1+ \ln N \right)^2 N^{-\min\{2, \, r(1+\alpha_0) \}}.
\end{align*}
Thus, we combine the above analysis to complete the proof.
\end{proof}

Next, we will establish the convergence result of the time-semidiscrete scheme, given by the following theorem.

\begin{theorem}\label{thm3.9}
Let $u^n$ and $U^n$ be the solution of \eqref{qwl03_new} and \eqref{qwl011}, respectively. It holds that
    \begin{align*}
       \|u^n-U^n\| \leq C(T) N^{-\min\{2-\alpha_0, \, r\alpha_0 \}}, \quad 1\leq n \leq N.
    \end{align*}
\end{theorem}
\begin{proof}
   First, we denote
\begin{align*}
  \rho^n = u^n - U^n, \quad 0\leq n \leq N.
\end{align*}
We subtract \eqref{qwl011} from \eqref{qwl010} to yield the error equations as follows
\begin{align}
    & D_N^{\alpha_0} \rho^n - \Delta \rho^n = - \Big( \omega_{n,1} \rho^{1} + \sum\limits_{j=2}^{n} \omega_{n,j} \rho^{j-1/2} \Big) + R^n,   \quad n\geq 1, \label{lkx01} \\
    & \rho^0  = 0. \label{lkx02}
\end{align}
Based on Remark \ref{rm2.4}, for \eqref{lkx01}--\eqref{lkx02} we similarly obtain
\begin{align*}
   \|\rho^n\| & \leq 2E_{\alpha_0}\!\left(2\lambda_{*} t_n^{\alpha_0}\right)
  \Big(  \|\rho^0\| + 2\max_{1\le j\le n}\sum_{l=1}^{j} P_{j-l}^{(j)} \|R^l\| \Big) \\
  & \leq C(T)\max_{1\le j\le n}\sum_{l=1}^{j} P_{j-l}^{(j)} \|R^l\|
     \leq C(T)\max_{1\le j\le n}\sum_{l=1}^{j} P_{j-l}^{(j)} \big( \|(R_1)^l\| + \|(R_2)^l\| \big),
\end{align*}
which follows Lemmas \ref{lem03} and \ref{lem05} to obtain
\begin{align*}
   \|\rho^n\| & \leq C(T) \Big[ N^{-\min\{2-\alpha_0, \, r\alpha_0\}} +  \left(1+ \ln N \right)^2 N^{-\min\{2, \, r(1+\alpha_0) \}}  \Big].
\end{align*}
We combine this with
\begin{align*}
   \left(1+ \ln N \right)^2 N^{-\min\{2, \, r(1+\alpha_0) \}} & = \left(1+ \ln N \right)^2 N^{-\alpha_0} N^{-\min\{2-\alpha_0, \, r\alpha_0 + (r-\alpha_0) \}} \\
  & = \left( \frac{ 1+ \ln N}{N^{\alpha_0/2}} \right)^2 N^{-\min\{2-\alpha_0, \, r\alpha_0 + (r-\alpha_0) \}} \\
  & \leq \left( \frac{ 1+ \ln N}{N^{\alpha_0/2}} \right)^2 N^{-\min\{2-\alpha_0, \, r\alpha_0 \}}
\end{align*}
and $\lim\limits_{N\rightarrow \infty}\left( \frac{ 1+ \ln N}{N^{\alpha_0/2}} \right)^2 =0$ to finish the proof of the theorem.
\end{proof}

\subsection{Analysis of fully discrete scheme}

In this section, we formulate and analyze a fully discrete Galerkin scheme corresponding to the temporal semi-discrete scheme \eqref{qwl011}--\eqref{qwl012}.

Denote a quasi-uniform partition of $\Omega$ with mesh size $h$. Let $S_h$ be the finite element space consisting of continuous piecewise linear functions defined on $\Omega$. Denote the Ritz projector $I_h: H^1_0(\Omega)\rightarrow S_h$ by
\begin{align}
     \big(\nabla (I_h \varphi - \varphi), \nabla \hat{\chi} \big) = 0 \quad \text{for} \;\; \hat{\chi} \in S_h \label{yq05}
\end{align}
with the following approximation property \cite{Tho}
\begin{align}
  \|\p_t^{\ell}(\varphi - I_h \varphi) t\|\leq C h^2  \|\p_t^{\ell}\varphi\|_{H^2(\Omega)}, \quad \ell=0,1. \label{yq06}
\end{align}

Based on the time-semidiscrete scheme \eqref{qwl011}--\eqref{qwl012}, the fully discrete   Galerkin finite element scheme reads: find $U^n_h\in S_h$ such that for $1 \le n \le N$,
\begin{align}
   \left(D_N^{\alpha_0} U^n_h, \hat{\chi} \right) + \left(\nabla U_h^{n}, \nabla\hat{\chi} \right)  =  - \left(I_N^{\omega} U^n_h, \hat{\chi} \right) + (F^n,\hat{\chi} )  \label{lkx001}
\end{align}
with $I_N^{\omega} U^n_h :=\big[ \omega_{n,1} U^{1}_h + \sum_{j=2}^{n} \omega_{n,j} U^{j-1/2}_h \big]$ and $U_h^0= I_hu_0$. For the sake of convenience in the subsequent analysis, we introduce
 \begin{align*}
     & u(x,t_n)- U_h^n =\zeta^n - \eta^n, \quad 0\leq n \leq N,
 \end{align*}
in which
\begin{align*}
    & \zeta^n = I_h u(x,t_n) - U_h^n \in S_h, \quad \eta^n = I_h u(x,t_n) - u(x,t_n).
\end{align*}
Combining \eqref{qwl010} and \eqref{lkx001}, and applying \eqref{yq05}, we obtain
\begin{align}
   \left(D_N^{\alpha_0}\zeta^n, \hat{\chi} \right) + \left(\nabla \zeta^{n}, \nabla\hat{\chi} \right)  = - \left(I_N^{\omega}\zeta^n,  \hat{\chi} \right)  +  \left(R^{n} + D_N^{\alpha_0}\eta^n + I_N^{\omega}\eta^n, \, \hat{\chi} \right).   \label{lkx002}
\end{align}
We next establish the stability and derive the error estimates for the fully discrete finite element scheme.

\begin{theorem}  Let $U_h^n$ be the solution of the fully discrete scheme \eqref{lkx001}. Then the following stability result holds
  \begin{align}
       \|U_h^n\| \leq 2E_{\alpha_0}\!\left(2\lambda_{*} t_n^{\alpha_0}\right)
  \Big( \|U_h^0\| + 2\max_{1\le j\le n}\sum_{l=1}^{j} P_{j-l}^{(j)} \|F^l\| \Big), \quad 1\le n\le N.    \label{lkx003}
    \end{align}
Furthermore, under the regularity result \eqref{regu}, the following error estimate holds for $1\leq n \leq N$
    \begin{align}
         \|u^n-U_h^n\| \leq C(T)  \left( N^{-\min\{2-\alpha_0, \, r\alpha_0 \}} + h^2 \right). \label{lkx004}
    \end{align}
\end{theorem}

\begin{proof}
  We first follow the proof of Remark \ref{rm2.4} and then take $\hat{\chi}=U_h^n$ in \eqref{lkx001} to obtain \eqref{lkx003}. We then choose $\hat{\chi} =\zeta^n$ in \eqref{lkx002} and follow  \eqref{qqww01}--\eqref{qqww02} to conclude that
 \begin{align*}
     D_N^{\alpha_0}(\|\zeta^n\|)^2 & \leq  \sum\limits_{j=1}^{n-1} \lambda_{n-j} \| \zeta^{j}\|^2 + 2 \big(\|R^n + D_N^{\alpha_0}\eta^n + I_N^{\omega}\eta^n\| \big) \|\zeta^n\| \\
 & \leq  \sum\limits_{j=1}^{n-1} \lambda_{n-j} \| \zeta^{j}\|^2 + 2 \big(\|R^n\| + \|D_N^{\alpha_0}\eta^n\| + \|I_N^{\omega}\eta^n\| \big) \|\zeta^n\|,
\end{align*}
which uses Lemma \ref{lem02} with $\xi^n =2 \|R^n\|$ and $\vartheta^n = 2 \big( \|D_N^{\alpha_0}\eta^n\| + \|I_N^{\omega}\eta^n\| \big)$ to yield
\begin{equation*}
  \|\zeta^n\| \le 2E_{\alpha_0}\!\left(2\lambda_{*} t_n^{\alpha_0}\right)
  \Big( \|\zeta^0\| + 2\max_{1\le j\le n}\sum_{l=1}^{j} P_{j-l}^{(j)} \|R^l\|
     + 2\beta_{1+\alpha_0}(t_n)\max_{1\le j\le n} \big( \|D_N^{\alpha_0}\eta^j\| + \|I_N^{\omega}\eta^j\| \big)  \Big).
\end{equation*}
We then apply Lemma \ref{thm3.9} to obtain
\begin{equation*}
  \|\zeta^n\| \le C(T)
  \Big[ N^{-\min\{2-\alpha_0, \, r\alpha_0 \}} + \max_{1\le m\le n} \big( \|D_N^{\alpha_0}\eta^m\| + \|I_N^{\omega}\eta^m\| \big)  \Big].
\end{equation*}
We follow \eqref{yq06} and our previous work \cite[Equations (55)--(59)]{Qiu} to get
\begin{align*}
  \max_{1\le m\le n}  \|D_N^{\alpha_0}\eta^m\| & \leq Ch^2.
\end{align*}
Consequently, the proof is reduced to deriving an estimate for $\|I_N^{\omega}\eta^m\|$ with $1\leq m \leq n$. Then, \eqref{yq06} and \eqref{mm03} give
\begin{align*}
    \|I_N^{\omega}\eta^m\| & \leq |\omega_{m,1}| \|\eta^1\| + \sum_{j=2}^{m} |\omega_{m,j}| \| \eta^{j-1/2} \|  \leq |\omega_{m,1}| \|\eta^1\| + \sum_{j=2}^{m} |\omega_{m,j}| \frac{\|\eta^{j}\| + \|\eta^{j-1}\|}{2} \\
     & \leq \max\limits_{1\leq j \leq m} \|\eta^j\| \Big( \sum_{j=1}^{m} |\omega_{m,j}| \Big)  \leq C h^{2}   \left\|u \right\|_{H^2} \int_{0}^{t_m} (t_m-s)^{-\alpha^*}ds \leq Ch^2.
\end{align*}
Invoking the above three estimates together with \eqref{yq06} yields \eqref{lkx004}, and hence completes the proof.
\end{proof}

 \section{Numerical experiment}\label{sec6}

 In this section, we will provide several numerical examples to validate our theoretical analysis. Without loss of generality, we set $T=1$ and $\Omega=(0,1)^d$. To better examine the space-time convergence rate, we choose $r=\frac{2-\alpha_0}{\alpha_0}$ in \eqref{mesh}. To avoid the impact of rounding errors, we use \cite[Algorithm 3.1]{Quan24ANM} to compute the coefficients in the nonuniform L1 formula \eqref{qwl04_new}.
 
 To measure the convergence of proposed methods, we apply the two-mesh principle \cite[p.~107]{Farrell} to denote the discrete $L^2$ errors and convergence rates as follows
 \begin{align*}
  & \text{Err}_{\tau}^d(\tau,h)
 = 
 \left(
 h^d
 \sum_{\boldsymbol{i}\in\mathcal I_h}
 \left|
 U^N_{\boldsymbol{i}}(\tau,h)
 -
 U^{2N}_{\boldsymbol{i}}\!\left(\tfrac{\tau}{2},h\right)
 \right|^2
 \right)^{1/2}, \quad \text{Rate}_{\tau}^d = \log\left( \frac{\text{Err}_{\tau}^d(2\tau,h)}{\text{Err}_{\tau}^d(\tau,h)} \right), \\
 & \text{Err}_{h}^d(\tau,h)
 = 
 \left(
 h^d
 \sum_{\boldsymbol{i}\in\mathcal I_h}
 \left|
 U^N_{\boldsymbol{i}}(\tau,h)
 -
 U^{N}_{2\boldsymbol{i}}\!\left(\tau,\tfrac{h}{2}\right)
 \right|^2
 \right)^{1/2}, \quad \text{Rate}_{h}^d = \log\left( \frac{\text{Err}_{h}^d(\tau,2h)}{\text{Err}_{h}^d(\tau,h)} \right), 
 \end{align*}
 where
 \[
 \boldsymbol{i} = (i_1,\ldots,i_d),
 \quad
 \mathcal I_h
 := \left\{
 \boldsymbol{i} \in \mathbb{N}^d :
 1 \le i_\ell \le J-1,\ \ell = 1,\ldots,d
 \right\}.
 \]
In the following examples, we fix $J=32$ to test the time convergence rate and accordingly fix $N=128$ to test the spatial convergence rate, respectively.

%\subsection{Convergence test for problem (\ref{VEPDE}) }

\textbf{Example 1.} For $d=1$, we set $u_0=\sin (\pi x)$ and $f\equiv 1$, and the variable exponent $\alpha(t)=\alpha_0+\frac{1}{5}\sin (t)$.
We present the convergence results in Tables~\ref{tab1}--\ref{tab2}, which demonstrate that the scheme exhibits the temporal convergence rate of order $2-\alpha_0$ and the spatial convergence rate of second order. These results are consistent with the theoretical analysis.

\begin{table}
    \center  \small
    \caption{Errors and convergence rates in time   for Example 1.} \label{tab1}
%  \resizebox{\textwidth}{!}{
    \begin{tabular}{cccccccccccc}
      \toprule
    & \multicolumn{2}{c}{$\alpha_0 =0.3$} &
     &\multicolumn{2}{c}{$\alpha_0 =0.5$} &
     &\multicolumn{2}{c}{$\alpha_0 =0.75$}\\
   \cmidrule{2-3}  \cmidrule{5-6} \cmidrule{8-9}
       $N$  & $\text{Err}_{\tau}^1(\tau,h)$ & $\text{Rate}_{\tau}^1$ & $N$  & $\text{Err}_{\tau}^1(\tau,h)$ & $\text{Rate}_{\tau}^1$ & $N$  & $\text{Err}_{\tau}^1(\tau,h)$ & $\text{Rate}_{\tau}^1$\\
      \midrule
        32   &  $3.2820 \times 10^{-5}$ &  *      &  64    & $2.8642 \times 10^{-5}$  &  *    &  128    & $3.6011 \times 10^{-5}$  &  *    \\
        64   &  $9.5333 \times 10^{-6}$ &  1.78   &  128   & $9.4374 \times 10^{-6}$  &  1.60 &  256    & $1.4997 \times 10^{-5}$  &  1.26 \\
        128   &  $2.7686 \times 10^{-6}$ &  1.78   &  256   & $3.1504 \times 10^{-6}$  &  1.58 &  512   & $6.2907 \times 10^{-6}$  &  1.25 \\
        256   &  $7.9358 \times 10^{-7}$ &  1.80   &  512   & $1.0599 \times 10^{-6}$  &  1.57 &  1024   & $2.6457 \times 10^{-6}$  &  1.25 \\
      \bottomrule
    \end{tabular}
%    }
\end{table}

\begin{table}
    \center \small
    \caption{Errors and convergence rates in space for Example 1.} \label{tab2}
     %\vspace{-0.2in}
%    \resizebox{\textwidth}{!}{
  \begin{tabular}{ccccccccccccccccccc}
      \toprule
    & \multicolumn{2}{c}{$\alpha_0 =0.2$} &
     &\multicolumn{2}{c}{$\alpha_0 =0.4$} &
     &\multicolumn{2}{c}{$\alpha_0 =0.6$}\\
   \cmidrule{2-3}  \cmidrule{5-6} \cmidrule{8-9}
       $J$  & $\text{Err}_{h}^1(\tau,h)$ & $\text{Rate}_{h}^1$ & $J$  & $\text{Err}_{h}^1(\tau,h)$ & $\text{Rate}_{h}^1$& $J$  & $\text{Err}_{h}^1(\tau,h)$ & $\text{Rate}_{h}^1$\\
      \midrule
        32   &  $5.3413 \times 10^{-4}$ &    *    &  32    & $5.2267 \times 10^{-4}$  &    *    &  32    & $5.0781 \times 10^{-4}$  &    *    \\
        64   &  $1.3529 \times 10^{-4}$ &  1.98   &  64    & $1.3246 \times 10^{-4}$  &  1.98   &  64    & $1.2880 \times 10^{-4}$  &  1.98 \\
        128  &  $3.4038 \times 10^{-5}$ &  1.99   &  128   & $3.3336 \times 10^{-5}$  &  1.99   &  128   & $3.2426 \times 10^{-5}$  &  1.99 \\
        256  &  $8.5362 \times 10^{-6}$ &  2.00   &  256   & $8.3613 \times 10^{-6}$  &  2.00   &  256   & $8.1345 \times 10^{-6}$  &  2.00 \\
      \bottomrule
    \end{tabular}
%    }
\end{table}

\textbf{Example 2.}  Let $d=2$, $u_0=\sin (\pi x)\sin (\pi y)$ and $f\equiv 1$. We choose 
\begin{align*}
  \alpha(t)=  \alpha_T + \big( \alpha_0 - \alpha_T \big) \left( 1-\frac{t^2}{T^2} \right),
\end{align*}
which satisfies $\alpha(0) =\alpha_0$ and $\alpha(T)=\alpha_T$.

Tables~\ref{tab3}--\ref{tab4} present the convergence results, which show that the scheme exhibits a temporal convergence of order $2-\alpha_0$ and a second-order spatial convergence. These results substantiate our theoretical findings.

\begin{table}
    \center  \small
    \caption{Errors and convergence rates in time  for Example 2.} \label{tab3}
%  \resizebox{\textwidth}{!}{
    \begin{tabular}{cccccccccccc}
      \toprule
    & \multicolumn{2}{c}{$\alpha_0=0.4$, $\alpha_T=0.8$} &
     &\multicolumn{2}{c}{$\alpha_0=0.6$, $\alpha_T=0.4$} &
     &\multicolumn{2}{c}{$\alpha_0=0.8$, $\alpha_T=0.1$}\\
   \cmidrule{2-3}  \cmidrule{5-6} \cmidrule{8-9}
       $N$  & $\text{Err}_{\tau}^2(\tau,h)$ & $\text{Rate}_{\tau}^2$ & $N$  & $\text{Err}_{\tau}^2(\tau,h)$ & $\text{Rate}_{\tau}^2$ & $N$  & $\text{Err}_{\tau}^2(\tau,h)$ & $\text{Rate}_{\tau}^2$\\
      \midrule
        64   &  $4.9229 \times 10^{-6}$ &  *      &  128    & $1.8293 \times 10^{-6}$  &  *    &  256    & $1.0019 \times 10^{-5}$  &  *    \\
        128   &  $1.6695 \times 10^{-6}$ &  1.56   &  256   & $6.9275 \times 10^{-7}$  &  1.40 &  512    & $4.6321 \times 10^{-6}$  &  1.11 \\
        256   &  $5.6593 \times 10^{-7}$ &  1.56   &  512   & $2.6367 \times 10^{-7}$  &  1.39 &  1024   & $2.1021 \times 10^{-6}$  &  1.14 \\
        512   &  $1.9049 \times 10^{-7}$ &  1.57   &  1024   & $1.0035 \times 10^{-7}$  &  1.39 &  2048   & $9.4478 \times 10^{-7}$  &  1.15 \\
      \bottomrule
    \end{tabular}
%    }
\end{table}

\begin{table}
    \center  \small
    \caption{Errors and convergence rates in space   for Example 2.} \label{tab4}
     %\vspace{-0.2in}
%    \resizebox{\textwidth}{!}{
  \begin{tabular}{ccccccccccccccccccc}
      \toprule
    & \multicolumn{2}{c}{$\alpha_0=0.3$, $\alpha_T=0.7$} &
     &\multicolumn{2}{c}{$\alpha_0=0.5$, $\alpha_T=0.3$} &
     &\multicolumn{2}{c}{$\alpha_0=0.7$, $\alpha_T=0.1$}\\
   \cmidrule{2-3}  \cmidrule{5-6} \cmidrule{8-9}
       $J$  & $\text{Err}_{h}^2(\tau,h)$ & $\text{Rate}_{h}^2$ & $J$  & $\text{Err}_{h}^2(\tau,h)$ & $\text{Rate}_{h}^2$& $J$  & $\text{Err}_{h}^2(\tau,h)$ & $\text{Rate}_{h}^2$\\
      \midrule
        32   &  $4.1987 \times 10^{-5}$ &    *    &  32    & $9.0962 \times 10^{-5}$  &    *    &  32    & $1.0269 \times 10^{-4}$  &    *    \\
        64   &  $1.0644 \times 10^{-5}$ &  1.98   &  64    & $2.2952 \times 10^{-5}$  &  1.99   &  64    & $2.5893 \times 10^{-5}$  &  1.99 \\
        128  &  $2.6696 \times 10^{-6}$ &  2.00   &  128   & $5.7511 \times 10^{-6}$  &  2.00   &  128   & $6.4866 \times 10^{-6}$  &  2.00 \\
        256  &  $6.6790 \times 10^{-7}$ &  2.00   &  256   & $1.4386 \times 10^{-6}$  &  2.00   &  256   & $1.6225 \times 10^{-6}$  &  2.00 \\
      \bottomrule
    \end{tabular}
%    }
\end{table}

\section{Concluding remarks}

 In this work, we investigate the variable-exponent subdiffusion model under nonuniform temporal meshes and derive its optimal error estimate. 
By reformulating the original model via perturbation methods, the Caputo derivative and convolution terms in the reformulated model are discretized by the L1 scheme combined with an interpolation quadrature rule. 
Rigorous numerical analysis establishes a temporal convergence rate of 
$O\big(N^{-\min\{2-\alpha_0,\, r\alpha_0\}}\big)$,
which  improves the existing results in \cite{Zheng} for $r \ge \frac{2-\alpha_0}{\alpha_0}$. Numerical experiments confirm the accuracy and effectiveness of the theoretical estimates.

One interesting topic is to investigate the high-order numerical scheme for  \eqref{eq1.1}--\eqref{eq1.2}. The L2-1$_{\sigma}$ scheme is proposed in  \cite{Alik} to approximate the time-fractional diffusion equation. In principle, one could combine \eqref{regu}, the approach in  \cite{Alik} and follow  \cite{Liao1} to analyze the resulting high-order scheme for the reformulated model \eqref{VEPDE} to obtain the temporal accuracy of order $O\big(N^{-\min\{r\alpha_0, \,2\}}\big)$, the details of which will be investigated in the near future.

%%%%%%%%%%%%%%%%%%%%%%%%%%
\appendix

\setcounter{equation}{0}
\renewcommand{\theequation}{A.\arabic{equation}}
\setcounter{theorem}{0}
\renewcommand{\thetheorem}{A.\arabic{theorem}}

\section*{Appendix: Proof of (\ref{conv3})}
We here establish the convergence of the time-discrete scheme \eqref{q004}. With the notation $\tilde\rho^n:=u^n-U^n$ for $0\leq n \leq N$, we can give the following error equations for \eqref{q004} 
 \begin{align}
       & D_N^{\alpha_0} \tilde\rho^n - \Delta \tilde\rho^n = -I_N^{\tilde{\omega}}\mathcal{A}\tilde\rho^n  +  \tilde{R}^n, \quad n = 1,2,\cdots,N, \label{q006} \\
       & \tilde\rho^0 = 0, \label{q007}
   \end{align}
where $\tilde{R}^n:=(R_1)^n-(\tilde{R}_2)^n$.
 The error $(R_1)^n$ bounded by \eqref{qwl06_new} and the quadrature errors $(\tilde{R}_2)^n:=\int_{0}^{t_n}\mathcal{K}(t_n-s)\mathcal{A}u(x,s) ds- I_N^{\tilde{\omega}}\mathcal{A}u^n$ are given via
 \begin{equation}\label{q003}
   \begin{split}
      (\tilde{R}_2)^1  &  = \int_{0}^{t_1} \mathcal{K}(t_1-s) \Big( \int_{t_1}^{s} \p_{\theta} \mathcal{A} u(x,\theta)d\theta \Big)  ds, \\
      (\tilde{R}_2)^n  &  = \int_{0}^{t_1} \mathcal{K}(t_n-s) \Big( \int_{t_1}^{s} \p_{\theta} \mathcal{A} u(x,\theta)d\theta \Big)  ds \\
       & \quad +  \sum\limits_{j=2}^{n}\int_{t_{j-1}}^{t_j}\mathcal{K}(t_n-s) \Big(  \int_{t_{j-1}}^{t_j} K(s;\theta) \p_{\theta}^2\mathcal{A} u(x,\theta)d\theta \Big)  ds\\
       & =: (\tilde{R}_{21})^n + (\tilde{R}_{22})^n, 
       \quad n \geq 2,
   \end{split}
 \end{equation}
 where the Peano kernel $K(s;\theta)$ given by \eqref{ker}. Next, we will prove the convergence of the scheme \eqref{q004}. Before that, we give a helpful lemma for the estimate of $(\tilde{R}_2)^n$ with $n\geq 1$.

 \begin{lemma}\label{lem5.1}
    Under the regularity assumption \eqref{regu}, it hold that
    \begin{align*}
        &\left\|(\tilde{R}_2)^1\right\| \leq CN^{-r(1+\alpha_0-\mu)},  \\
        &\left\|(\tilde{R}_2)^n\right\| \leq Ct_{n}^{-\alpha_0}\, \Big[  N^{-\min\left\{ r(1+2\alpha_0-\mu), \, 2 \right\}} + \Phi^N(\alpha_0,\mu,r) \Big], \quad 2\leq n \leq N,
    \end{align*}
 where the notation $\Phi^N(\alpha_0,\mu,r)$ defined by 
 \begin{align*}
     \Phi^N(\alpha_0,\mu,r) :=
      \begin{cases}
       \displaystyle  N^{-\min\{r(1+2\alpha_0-\mu), \, 2\}},  &  r(1+\alpha_0)>2, \nonumber \vspace{1mm} \\
       \displaystyle  N^{-\min\{r(1+2\alpha_0-\mu), \, r(1+\alpha_0)\}} \ln N, &  r(1+\alpha_0)=2, \nonumber \vspace{1mm} \\
       \displaystyle  N^{-\min\{r(1+2\alpha_0-\mu), \, r(1+\alpha_0)\}}, &  r(1+\alpha_0)<2.
      \end{cases} 
 \end{align*}
 \end{lemma}
 \begin{proof}
    For \eqref{q003}, we shall apply the regularity result \eqref{regu} to deduce the desired result. \par
   \textbf{I: Estimate of $\|(\tilde{R}_2)^1\|$.} We employ \eqref{regu} to get
 \begin{align*}
    \|(\tilde{R}_2)^1\| & \leq  \int_{0}^{t_1} |\mathcal{K}(t_1-s)|  \Big( \int_{s}^{t_1} \|\p_{\theta} \mathcal{A} u(\cdot,\theta) \| d\theta \Big) ds \\
 & \leq C_0 \int_{0}^{t_1}  \theta^{\alpha_0-1} d\theta  \int_{0}^{t_1}(t_1-s)^{-\mu} ds    \leq Ct_1^{\alpha_0} 
\int_{0}^{t_1}(t_1-s)^{-\mu} ds \\
 & \leq Ct_1^{\alpha_0} t_1^{1-\mu} = Ct_1^{1+\alpha_0-\mu} \leq C N^{-r(1+\alpha_0-\mu)}.
 \end{align*}

 \textbf{II: Estimate of $\|(\tilde{R}_{21})^n\|$.} For $n\geq 2$, we have $t_n-t_1\geq c \, t_n$,
satisfying for $s\in (0,t_1)$
  \begin{align*}
     |\mathcal{K}(t_n-s)| & \leq C(t_n-s)^{-\mu} \leq C(t_n-t_1)^{-\mu} \leq C t_n^{-\mu},
  \end{align*}
thus we employ \eqref{regu} to get
\begin{align*}
   \|(\tilde{R}_{21})^n\| & \leq  \int_{0}^{t_1} |\mathcal{K}(t_n-s)|  \Big( \int_{s}^{t_1} \|\p_{\theta} u(\cdot,\theta) \| d\theta \Big) ds \leq C t_n^{-\mu} t_1  \int_{0}^{t_1}  \theta^{\alpha_0-1} d\theta \\
&    \leq Ct_n^{-\mu} t_1^{\alpha_0+1}  \leq Ct_n^{-\alpha_0} t_n^{\alpha_0-\mu}  N^{-r(1+\alpha_0)} \leq Ct_n^{-\alpha_0} \frac{n^{r(\alpha_0-\mu)}}{N^{r(\alpha_0-\mu)}} N^{-r(1+\alpha_0)}  \\
& \leq Ct_n^{-\alpha_0} \, N^{-\min\{r(1+\alpha_0), \, r(1+2\alpha_0-\mu)\}}.
\end{align*}

 \textbf{III: Estimate of $\|(\tilde{R}_{22})^n\|$.} For $n\geq 2$, we rewrite $(\tilde{R}_{22})^n$ as
 \begin{align*}
        (\tilde{R}_{22})^n  & = \int_{t_{n-1}}^{t_n}\mathcal{K}(t_n-s) \Big(  \int_{t_{n-1}}^{t_n} K(s;\theta) \p_{\theta}^2 \mathcal{A}u(x,\theta)d\theta \Big)  ds \\
       & \quad + \sum\limits_{j=2}^{n-1}\int_{t_{j-1}}^{t_j}\mathcal{K}(t_n-s) \Big(  \int_{t_{j-1}}^{t_j} K(s;\theta) \p_{\theta}^2\mathcal{A} u(x,\theta)d\theta \Big)  ds =: (\tilde{R}_{221})^n + (\tilde{R}_{222})^n.
 \end{align*}
 We first bound $(\tilde{R}_{221})^n$ via \eqref{regu} and \eqref{ker}
 \begin{align*}
    \|(\tilde{R}_{221})^n\| & \leq C\tau_n \int_{t_{n-1}}^{t_n} |\mathcal{K}(t_n-s)| \Big( \int_{t_{n-1}}^{t_n} \|\p_{\theta}^2 \mathcal{A}u(\cdot,\theta)\|d\theta  \Big) ds \\
    & \leq C\tau_n^2 t_{n-1}^{\alpha_0-2} \int_{t_{n-1}}^{t_n} (t_n-s)^{-\mu} ds \leq C\tau_n^{3-\mu} t_{n}^{\alpha_0-2} = Ct_n^{-\alpha_0}\tau_n^{3-\mu} t_{n}^{2\alpha_0-2},
 \end{align*}
 which follows
 \begin{align*}
      \tau_n^{3-\mu} t_{n}^{2\alpha_0-2} & \leq C  \frac{n^{(r-1)(3-\mu)}}{N^{r(3-\mu)}} \Big(\frac{n}{N}\Big)^{2r(\alpha_0-1)} = C \frac{n^{r(1+2\alpha_0-\mu)-(3-\mu)} }{N^{r(1+2\alpha_0-\mu)-(3-\mu)}} N^{-(3-\mu)}
 \end{align*}
 to arrive at
 \begin{align*}
    \|(\tilde{R}_{221})^n\| & \leq C t_{n}^{-\alpha_0}  N^{-\min\left\{ 3-\mu,  \, r(1+2\alpha_0-\mu) \right\}}.
 \end{align*}
 We then bound $(\tilde{R}_{222})^n$ and rewrite it as
 \begin{align*}
    (\tilde{R}_{222})^n & = \sum\limits_{j=\lceil n/2 \rceil+1}^{n-1}\int_{t_{j-1}}^{t_j}\mathcal{K}(t_n-s) \Big(  \int_{t_{j-1}}^{t_j} K(s;\theta) \p_{\theta}^2 \mathcal{A}u(x,\theta)d\theta \Big)  ds \\
  & \quad + \sum\limits_{j=2}^{\lceil n/2 \rceil} \int_{t_{j-1}}^{t_j}\mathcal{K}(t_n-s) \Big(  \int_{t_{j-1}}^{t_j} K(s;\theta) \p_{\theta}^2 \mathcal{A}u(x,\theta)d\theta \Big)  ds =: (\tilde{R}_{2221})^n + (\tilde{R}_{2222})^n.
 \end{align*}
 We observe that, for $2\le j \leq n-1$,
 \begin{align*}
    \Big\| \int_{t_{j-1}}^{t_j} & \mathcal{K}(t_n-s) \Big(  \int_{t_{j-1}}^{t_j} K(s;\theta) \p_{\theta}^2 \mathcal{A}u(\cdot,\theta)d\theta \Big)  ds  \Big\| \\
    & \leq C\tau_j \int_{t_{j-1}}^{t_j}|\mathcal{K}(t_n-s)| \int_{t_{j-1}}^{t_j}\|\p_{\theta}^2 \mathcal{A}u(\cdot,\theta)\| d\theta  ds \\
    & \leq C\tau_j \int_{t_{j-1}}^{t_j} (t_n-s)^{-\mu}   ds  \int_{t_{j-1}}^{t_j}\theta^{\alpha_0-2} d\theta
    \leq C\tau_j^2 t_j^{\alpha_0-2} \int_{t_{j-1}}^{t_j} (t_n-s)^{-\mu} ds.
 \end{align*}
 Hence, for $\lceil n/2 \rceil +1 \leq j \leq n-1$, we get $t_j\geq 2^{-r}t_n$, which gives
 \begin{align*}
     \|(\tilde{R}_{2221})^n\|
   & \leq  C\sum\limits_{j=\lceil n/2 \rceil+1}^{n-1}  \tau_j^2 t_{j-1}^{\alpha_0-2} \int_{t_{j-1}}^{t_j} (t_n-s)^{-\mu} ds  \\
   & \leq  C \tau_n^2 \, (t_{\lceil n/2 \rceil})^{\alpha_0-2}  \sum\limits_{j=\lceil n/2 \rceil+1}^{n-1}  \int_{t_{j-1}}^{t_j} (t_n-s)^{-\mu} ds \\
   & \leq  C \tau_n^2 \, 2^{r(2-\alpha_0)}  t_{n}^{\alpha_0-2} \int_{t_{\lceil n/2 \rceil}}^{t_{n-1}} (t_n-s)^{-\mu} ds \leq  C  t_{n}^{-\alpha_0} \tau_n^{2} \,  t_{n}^{2\alpha_0-1-\mu}  \\
   & \leq C t_n^{-\alpha_0} \, \frac{n^{r(1+2\alpha_0-\mu)-2}}{N^{r(1+2\alpha_0-\mu)-2}} N^{-2}   \leq C t_{n}^{-\alpha_0} \,  N^{-\min\left\{ r(1+2\alpha_0-\mu),\, 2 \right\}}.
 \end{align*}
 Then for $ 2 \leq j \leq \lceil n/2 \rceil $, we have $(t_n-t_j)^{-\mu} \leq  (1-2^{-r})^{-\mu} t_n^{-\mu} $, which implies
 \begin{align*}
     \|(\tilde{R}_{2222})^n\|
   & \leq  C\sum\limits_{j=2}^{\lceil n/2 \rceil}  \tau_j^2 t_{j-1}^{\alpha_0-2}  \int_{t_{j-1}}^{t_j} (t_n-s)^{-\mu} ds  \leq  C \sum\limits_{j=2}^{\lceil n/2 \rceil}  \tau_j (t_n-t_j)^{-\mu}  \tau_j^2 t_{j}^{\alpha_0-2}  \\
 & \leq C t_n^{-\mu} \sum\limits_{j=2}^{\lceil n/2 \rceil}  t_{j}^{\alpha_0-2}  \tau_j^3
 \leq C t_n^{-\mu} \sum\limits_{j=2}^{\lceil n/2 \rceil} \frac{j^{-r(2-\alpha_0)}}{N^{-r(2-\alpha_0)}} \frac{j^{3(r-1)}}{N^{3r}}.
 \end{align*}
This further gives
 \begin{align*}
     \|(\tilde{R}_{2222})^n\| & \leq  C t_n^{-\mu} N^{-r(1+\alpha_0)} \sum\limits_{j=2}^{\lceil n/2 \rceil} j^{r(1+\alpha_0)-3} \nonumber \\
     & \leq  C t_n^{-\alpha_0} \left(\frac{n}{N}\right)^{r(\alpha_0-\mu)} N^{-r(1+\alpha_0)} \sum\limits_{j=2}^{\lceil n/2 \rceil} j^{r(1+\alpha_0)-3} \nonumber \\
     &  \leq  C  \, t_n^{-\alpha_0}  \times
      \begin{cases}
       \displaystyle  \frac{n^{r(1+2\alpha_0-\mu)-2}}{N^{r(1+2\alpha_0-\mu)-2}} N^{-2},  &  r(1+\alpha_0)>2, \nonumber \vspace{2mm} \\
       \displaystyle  \frac{n^{r(1+2\alpha_0-\mu)-r(1+\alpha_0)}}{N^{r(1+2\alpha_0-\mu)-r(1+\alpha_0)}} N^{-r(1+\alpha_0)} \ln n, &  r(1+\alpha_0)=2, \nonumber \vspace{2mm} \\
       \displaystyle  \frac{n^{r(1+2\alpha_0-\mu)-r(1+\alpha_0)}}{N^{r(1+2\alpha_0-\mu)-r(1+\alpha_0)}} N^{-r(1+\alpha_0)}, &  r(1+\alpha_0)<2
      \end{cases} \nonumber \vspace{2mm} \\
     &  \leq   C \, t_n^{-\alpha_0} \, \Phi^N(\alpha_0,\mu,r), 
 \end{align*}
 where we estimated $ \sum\limits_{j=2}^{\lceil n/2 \rceil} j^{r(1+\alpha_0)-3}$ by \eqref{bs} and note that
 \begin{align*}
     \Phi^N(\alpha_0,\mu,r) \leq C N^{-\min\{r(1+2\alpha_0-\mu), \, r(1+\alpha_0), \, 2 \}} (1+ \ln N).
 \end{align*}
 Based on the above analysis, we complete the proof.
 \end{proof}

 Therefore, we further yield the following lemma.

 \begin{lemma} \label{lem5.2}
   Under Lemma \ref{lem5.1} and Lemma \ref{lem01}, we get
    \begin{align*}
        \sum_{n=1}^{m}  P_{m-n}^{(m)} \left\|(\tilde{R}_2)^n\right\| \leq C (1+ \ln N ) \, N^{-\min\{r(1+2\alpha_0-\mu), \, r(1+\alpha_0), \, 2 \}}, \quad 1\leq  m \leq N.
    \end{align*}
 \end{lemma}
 \begin{proof}
 We first apply Lemma \ref{lem01} (i) and Lemma \ref{lem5.1} to get
 \begin{align*}
     P_{m-1}^{(m)} \left\|(\tilde{R}_2)^1\right\| \leq C\tau_1^{\alpha_0} N^{-r(1+\alpha_0-\mu)} \leq C N^{-r(1+2\alpha_0-\mu)}.
 \end{align*}
 Then we have
 \begin{align*}
     \sum_{n=2}^{m}  P_{m-n}^{(m)} \left\|(\tilde{R}_2)^n\right\| \leq C  \sum_{n=2}^{m}  P_{m-n}^{(m)} t_{n}^{-\alpha_0} \Big[ N^{-\min\{ 2, \,  r(1+\alpha_0) \}}  + \Phi^N(\alpha_0,\mu,r) \Big],
 \end{align*}
 which follows \eqref{bbb} to obtain
 \begin{align*}
   \sum_{n=2}^{m}  P_{m-n}^{(m)} \left\|(\tilde{R}_2)^n\right\| & \leq  C \Big[ N^{-\min\{ 2, \,  r(1+\alpha_0) \}}  +  N^{-\min\{r(1+2\alpha_0-\mu), \, r(1+\alpha_0), \, 2 \}} (1+ \ln N) \Big] \\
 & \leq CN^{-\min\{r(1+2\alpha_0-\mu), \, r(1+\alpha_0), \, 2 \}} (1+ \ln N).
 \end{align*}
 Therefore, the proof is finished.
 \end{proof}

\noindent\emph{\textbf{Proof of \eqref{conv3}.}} Following \eqref{stabi}, we use \eqref{q006}--\eqref{q007} to similarly obtain
   \begin{align*}
     \|\tilde\rho^m\| & \leq C(T) \Big( \|\tilde\rho^0\| + \max\limits_{1\leq n \leq m}  \sum_{j=1}^{n} P^{(n)}_{n-j} \|\tilde{R}^j\|  \Big) \\
    & \leq C(T)\max\limits_{1\leq n \leq m}  \sum_{j=1}^{n} P^{(n)}_{n-j} \big( \|(R_1)^j\| + \|(\tilde{R}_2)^j\| \big),  \quad 1\leq m \leq N,
 \end{align*}
 which follows Lemma \ref{lem03} and Lemma \ref{lem5.2} to arrive at
 \begin{align*}
   \|\tilde\rho^m\| & \leq C(T) N^{-\min\{2-\alpha_0, \, r\alpha_0 \} }   + C(T) N^{-\min\{r(1+2\alpha_0-\mu), \, r(1+\alpha_0), \, 2 \}} (1+ \ln N).
 \end{align*}
 We then apply $\displaystyle \frac{1+ \ln N}{N^{\alpha_0}}\leq C$ to yield
 \begin{align*}
    & N^{-\min\{r(1+2\alpha_0-\mu), \, r(1+\alpha_0), \, 2 \}} (1+ \ln N) \\
    & \qquad = N^{-\min\{r(1+\alpha_0-\mu)+(r-1)\alpha_0, \, r\alpha_0+(r-\alpha_0), \, 2-\alpha_0 \}} N^{-\alpha_0} (1+ \ln N) \\
    & \qquad \leq  N^{-\min\{r(1+\alpha_0-\mu), \, r\alpha_0, \, 2-\alpha_0 \}} N^{-\alpha_0} (1+ \ln N) \\
    & \qquad \leq C N^{-\min\{r\alpha_0+ r(1-\mu), \, r\alpha_0, \, 2-\alpha_0 \}}  \\
    & \qquad \leq C N^{-\min\{ r\alpha_0, \, 2-\alpha_0 \}}.
 \end{align*}
 We combine the above analysis to complete the proof.

\section*{Declaration}

\vskip 3mm
\noindent\textbf{Conflict of interest} The authors declare that they have no known competing financial interests or personal
relationships that could have appeared to influence the work reported in this paper.

%\vskip 3mm
%\noindent\textbf{Acknowledgments} The authors are grateful to the reviewers for their helpful suggestions to enhance the quality of the work.

\vskip 3mm
\noindent\textbf{Funding}
 This work was partially supported by the China Postdoctoral Science Foundation (No. 2024M762459), the Natural Science Foundation of Hubei Province (No. 2025AFB109), the Postdoctor Project of Hubei Province (No. 2025HBBSHCXB021), and the Postdoctoral Fellowship Program of CPSF (No. GZC20240938).

\vskip 3mm
\noindent\textbf{Data Availability} The datasets are available from the corresponding author upon reasonable request.

%%%%%%%%%%%%%%%%%%%%%%%%%%%%%%%%%%%%%%%%%%%%%%%%%%%%%%%%%%%%%%%%%%%%%%

\end{document}